\documentclass[10pt, oneside, a4paper]{article}
\usepackage{geometry}
\title{ Variation of height in an isogeny class over a function field}
\author{{ Richard {\sc Griffon}} \and { Samuel {\sc Le Fourn}} \and { Fabien {\sc Pazuki}}}	
\date{} 
\usepackage[english]{babel}
\usepackage[utf8]{inputenc}
\usepackage{lmodern}
\usepackage{amssymb,amsmath,amscd, amsthm}
\usepackage{mathrsfs}
\usepackage{tikz-cd}
\usepackage[all,cmtip]{xy}
\usepackage{enumerate}
\usepackage{graphicx}
\usepackage{bm}
\usepackage{titlesec}
\titleformat{\subsection}[runin]{\bfseries}{\thesubsection.}{0.2em}{}[.\hspace{0.4em}-- ]        
\titleformat{\section}{\center\Large\bfseries}{\thesection.}{0.2em}{}[]
\usepackage{url}
\definecolor{BrickRed}{RGB}{153, 0, 0}
\usepackage[colorlinks=true, citecolor= orange, linkcolor= BrickRed!80, urlcolor=black!60, pagebackref]{hyperref}
\usepackage{scalerel}
\newcounter{corocount}
\newtheorem{itheo}{Theorem}
 
\newtheorem{icoro}{Corollary}

\newtheorem*{itheono}{Theorem}

\newtheorem{theo}{Theorem}[section]
\newtheorem{coro}[theo]{Corollary}
\newtheorem{lemm}[theo]{Lemma}
\newtheorem{prop}[theo]{Proposition}
{\theoremstyle{definition}
	\newtheorem{defi}[theo]{Definition}
	\newtheorem{exple}[theo]{Example}
	\newtheorem{rema}[theo]{Remark}%
}	

\DeclareMathOperator{\DIVIS}{div}
\renewcommand{\div}{\DIVIS}			
\DeclareMathOperator{\Pic}{Pic}
\DeclareMathOperator{\Spec}{Spec}
\DeclareMathOperator{\DEGRE}{deg }
\renewcommand{\deg}{\DEGRE}

\newcommand{\R}{\ensuremath{\mathbb{R}}}

\newcommand{\Z}{\ensuremath{\mathbb{Z}}}

\renewcommand{\P}{\ensuremath{\mathbb{P}}}
\newcommand{\F}{\ensuremath{\mathbb{F}}}
\newcommand{\Ecal}{\mathcal{E}}			

\newcommand{\into}{\hookrightarrow}

\renewcommand{\bar}[1]{\ensuremath{\overline{#1}}}

\newcommand{\usep}{^\mathrm{sep}}
\newcommand{\ie}{\textit{i.e.}{}}

\newcommand{\Frob}{\mathrm{Fr}}
\newcommand{\Ver}{\mathrm{Ver}}

\newcommand{\hmod}{\mathrm{h}_{\mathsf{mod}}}
\newcommand{\hstab}{\mathrm{h}_{\mathsf{st}, K}}
\newcommand{\hdiff}{\mathrm{h}_{\mathsf{diff}}}
\newcommand{\hfal}{\mathrm{h}_{\mathrm{Fal}}}
\newcommand{\hd}{\mathrm{h}_{\mathrm{d}}}

\newcommand{\degins}{\mathrm{deg}_{\mathrm{ins}}}

\renewcommand{\O}{\mathcal{O}}

\newcommand{\Hcal}{\mathcal{H}}

\newcommand{\id}{\mathrm{id}}
\newcommand{\mmu}{\hm{\mu}}
\newcommand{\aalpha}{\hm{\alpha}}
\newcommand{\Lcal}{\mathcal{L}}
\newcommand{\Acal}{\mathcal{A}}
\newcommand{\Bcal}{\mathcal{B}}
\newcommand{\Gcal}{\mathcal{G}}
\newcommand{\Ocal}{\mathcal{O}}
\newcommand{\Fcal}{\mathcal{F}}
\newcommand{\Qcal}{\mathcal{Q}}
\newcommand{\FPPF}{\emph{fppf}}

\DeclareMathOperator{\Lie}{Lie}

\newcommand{\ins}{_{\mathrm{ins}}}
\newcommand{\sep}{_{\mathrm{sep}}}

\newcommand{\itemref}[1]{\hyperref[#1]{(\ref*{#1})}} 

\newsavebox{\pullback}
\sbox\pullback{%
	\begin{tikzpicture}%
		\draw (0.5ex, 0.5ex) -- (2ex,0.5ex);%
		\draw (2ex,0.5ex) -- (2ex,2ex);%
\end{tikzpicture}}

\begin{document}
	\pagestyle{plain}
	\maketitle 
	
	\noindent\hfill\rule{7cm}{0.5pt}\hfill\phantom{.}

	\paragraph{Abstract -- } 
	We give optimal 
        estimates on the variation of the differential and   modular heights within an isogeny class of abelian varieties defined over the function field of a curve (in any characteristic). 
        We also prove a parallelogram inequality for abelian varieties in this context, and deduce corollaries of these results.
	
	\medskip
	
	
	\smallskip
	\noindent{\it 2020 Math. Subj. Classification:}  
	11G10, 
	11G35, 
	11G50, 
	11R58, 
	14G17, 
	14G25, 
	14G40, 
	14K02, 
	14K15. 
	
	\noindent\hfill\rule{7cm}{0.5pt}\hfill\phantom{.}
	

\section*{Introduction}

\quad  
Let $k$ be a perfect field and let $C$ be a smooth projective geometrically irreducible curve over $k$, with function field $K := k(C)$. 
For any abelian variety $A$ defined over $K$, we let $\hdiff(A/K)$ denote the differential height of $A$ (see Definition~\ref{defi:hdiff}). 
In the context of abelian varieties over $K$, this measure of height plays the analogous r\^ole as that of the Faltings height in the setting of abelian varieties over number fields.
The general theme of this paper is to understand how the differential heights of two abelian varieties over $K$ linked by an isogeny are related. 

The following result of \cite{GriPaz} treats the case of abelian varieties of dimension $1$ over $K$, and also describes relation between the modular heights of isogenous elliptic curves.
We let $h :  K\to\R$ denote the (logarithmic) Weil height on~${K}$ \ie{}, $h(f)$ is the degree of $f\in K=k(C)$ viewed as a morphism $f:C\to\P^1$. 

\begin{itheono} (\cite{GriPaz}) 
Let $\phi :E \to E'$ be an isogeny between two non-isotrivial elliptic curves $E$, $E'$ defined over~$K$. If $K$ has positive characteristic, assume moreover that $E$ and $E'$ are semi-stable. 
We write $\widehat{\phi}$ for the dual isogeny, and $\deg\ins(-)$ denotes the inseparability degree of an isogeny (to be interpreted as $1$ if $K$ has characteristic~$0$). Then
\begin{enumerate}[(i)]
    \item The differential heights of $E$ and $E'$  satisfy
\begin{equation}\label{ieq:gripaz.hdiff}\tag{$H_d$} 
    \hdiff(E'/K) = \frac{\deg\ins(\phi)}{\degins(\widehat{\phi})}\, \hdiff(E/K)\,.
\end{equation}
\item The heights of the $j$-invariants of $E$ and $E'$ satisfy
\begin{equation}\label{ieq:gripaz.hmod}\tag{$H_m$}
    h\big(j(E')\big) 
    = \frac{\deg\ins(\phi)}{\degins(\widehat{\phi})}\, h\big(j(E)\big)\,.
\end{equation}
\end{enumerate}
\end{itheono}
 
These statements are Proposition 5.8 and Theorem A in \cite{GriPaz}, respectively.
The first (and main) goal of this paper is to obtain relations similar to \eqref{ieq:gripaz.hdiff} and \eqref{ieq:gripaz.hmod} for isogenies between higher-dimensional abelian varieties.
If the characteristic of $K$ is $0$ or if we restrict to isogenies of degree coprime to the characteristic, \eqref{ieq:gripaz.hdiff} generalizes rather directly (see the first assertion of Theorem \ref{itheo:hdiff} below). 
If~$K$ has positive characteristic $p$, however, the  higher-dimensional statement  and its proof are  more involved: in that case, indeed, the Frobenius isogeny is not necessarily `indecomposable', thus the inseparability degrees of an isogeny and its dual are, in general, too crude to capture the variation of the height.
\\

In order to state our first result, we thus introduce the following two notions.
For any given non-negative integers  $e, f$,  
 we say that an isogeny $\phi$ between abelian varieties over a field of positive characteristic $p$ is a {\it FV-isogeny of type $(e,f)$} if $\phi$ factorizes, in an arbitrary order, as a composition of $e$ many Frobenius isogenies, $f$ many Verschiebung isogenies,  and any number of separable isogenies whose dual is also separable. 
 (Every isogeny between non-isotrivial elliptic curves is an FV-isogeny by \cite[Proposition 4.7]{GriPaz}).
 
Secondly, define the {\it Frobenius height} $\delta_p(\phi)$ of an isogeny $\phi$ as the Frobenius height  of the kernel of the purely inseparable part $\phi\ins$ of $\phi$ (see Definition~\ref{defi:Frobheight}) \ie{}, if $G$ denotes the kernel of $\phi\ins$, $\delta_p(\phi)$ is the smallest integer $d\geq 0$ such that the $d$-th iterated Frobenius morphism $G\to G^{(p)}\to\dots\to G^{(p^d)}$ is trivial on $G$. 
 In particular, the Frobenius height of a FV-isogeny of type $(e,f)$ between ordinary abelian varieties 
satisfies $p^{\delta_p(\phi)}=\deg(\phi\ins)=p^e$ (see Lemma~\ref{lemm:frob.height.FV}).
For a general  isogeny $\phi:A\to B$, we have $p^{\delta_p(\phi)}\, |\, \deg(\phi\ins)$ and $\deg(\phi\ins) \,|\,p^{\dim(A)\, \delta_p(\phi)}$.  
 
With these definitions at hand, we can now state the first main theorem:
    \begin{itheo}\label{itheo:hdiff}
    Let $\phi : A\to B$ be an isogeny between two abelian varieties defined over $K$.
    \begin{enumerate}[\rm (1)]\setlength{\itemsep}{0em}
    \item\label{item:hdiff.a}
    If $K$ has characteristic $0$, or if the characteristic of $K$ does not divide $\deg\phi$, then
    \begin{equation}\label{ieq:char0}
        \hdiff(A/K) = \hdiff(B/K)\,.
    \end{equation}
    \item\label{item:hdiff.b} If $K$ has positive characteristic $p$, if $A$ is semi-stable over $K$, and if $\phi$ is a FV-isogeny of type $(e,f)$, then 
    \begin{equation}\label{ieq:charp.FVisog}
     \hdiff(B/K) = p^{e-f}\,\hdiff(A/K)\,.
	\end{equation}
	\item\label{item:hdiff.c}  If $K$ has positive characteristic $p$, if $A$ is ordinary and semi-stable over $K$, then 
    \begin{equation}\label{ieq:charp.gen}
        p^{-\delta_p(\widehat{\phi})}\,\hdiff(A/K) \leq \hdiff(B/K) \leq p^{\delta_p(\phi)}\,\hdiff(A/K)\,,   
    \end{equation}
    where $\delta_p(-)$ denotes the Frobenius height of an isogeny
    (see section~\ref{sec:Frob.height} for a definition).
    \end{enumerate}
    \end{itheo}

Equality \eqref{ieq:char0} is Theorem~\ref{theo:diffheight.bisep.isogenies}: it is an analogue of a result of Raynaud, who treats the number field case in \cite{Ray}.
Equality \eqref{ieq:charp.FVisog}, which is Corollary~\ref{VerFrob}, offers a direct extension  in higher dimension of \eqref{ieq:gripaz.hdiff} but only in a special case.
The most interesting and delicate part of the paper is inequality \eqref{ieq:charp.gen}, see Theorem~\ref{theo:hdiff.general}.
We observe that both inequalities in~\eqref{ieq:charp.gen} of Theorem~\ref{itheo:hdiff} are optimal in general: consider a separable isogeny whose dual is separable, and  the examples of the Frobenius isogeny and its dual.
In the second item \eqref{ieq:charp.FVisog}, the semi-stability assumption is necessary: see Example~\ref{exampleNonSS} for an instance of FV-isogeny of type $(1,0)$ which preserves the height.
In general, one cannot relax the ordinarity assumption in \eqref{ieq:charp.gen}: we refer to  sections~\ref{moretbabapara} and \ref{s:moretba.steroids} for a more detailed discussion.

When compared to the theorem of \cite{GriPaz}, the statement of Theorem \ref{itheo:hdiff} hints at the appearance of new phenomena for the variation of the height through isogenies between abelian varieties of dimension $>1$ in positive characteristic~$p$.
We illustrate some of these new behaviors with three examples of isogenies between abelian surfaces.
These suggest that there cannot exist a general relation as precise as equality \eqref{ieq:gripaz.hdiff} for  isogenies between arbitrary higher-dimensional abelian varieties.

Firstly, fix a non-isotrivial semi-stable elliptic curve $E$ defined over $K$ and let $\Frob_{E/K}:E\to E^{(p)}$ denote the Frobenius isogeny of $E$.
Set $A:=E\times E$ and $B=E\times E^{(p)}$. 
Proposition~\ref{hdiff_by_Frob} below shows that $\hdiff(E^{(p)}/K)=p\, \hdiff(E/K)$. 
Moreover, additivity of the differential height (see Proposition \ref{prop:hdiff.product}) yields $\hdiff(A/K) = 2\,\hdiff(E/K)$ and $\hdiff(B/K)= \hdiff(E/K) + \hdiff(E^{(p)}/K) = (p+1)\,\hdiff(E/K)$. 
The Frobenius isogeny $\Frob_{A/K}: A\to A^{(p)}$ of $A$, which coincides with $\Frob_{E/K}\times\Frob_{E/K}$, factors through the isogeny $\phi =\id_{E}\times \Frob_{E/K} : A\to B$ as follows:
\[ \begin{tikzcd} A \arrow[rrr, "\phi= \id_E\times \Frob_{E/K}"] \arrow[rrrrr, "\Frob_{A/K}"', bend right=20] &&& B \arrow[rr, dashed ] && A^{(p)}. \end{tikzcd}\]
We notice that, even though the isogeny $\phi$ has degree~$p$, 
the ratio $\hdiff(A/K)/\hdiff(B/K)=(p+1)/2$ is not a power of~$p$: this is in contrast with identity \eqref{ieq:gripaz.hdiff} in dimension $1$.
We thus need to deal with how isogenies which are factors of a Frobenius isogeny `interpolate' between the {\it additive} (in products) and {\it multiplicative} (through Frobenius) behaviors of the differential height. 
(See Example~\ref{exple:not.all.FV} for more details). 
 
Secondly, pick a non-isotrivial elliptic curve $E$ over $K$ and two finite subgroup schemes $H_1, H_2$ of $E$ with trivial intersection, and let $E':=E/(H_1+H_2)$ be the quotient elliptic curve.
From this data, Kani constructs in \cite{Kani97} an isogeny ${\phi : E\times E' \to E/H_1\times E/H_2}$ of degree $\deg(\phi)=(|H_1|+|H_2|)^2$.
It is possible to choose $H_1, H_2$ so that they both have order coprime to $p$, and $|H_1| + |H_2|$ is a power of $p$.
With these restrictions, one deduces from \eqref{ieq:gripaz.hdiff} that $\hdiff(E\times E') = \hdiff(E/H_1\times E/H_2)$. 
One also checks that the isogeny $\phi$ is of degree a power of $p$ and that $\phi$ is minimal (in the sense that it does not factor through any non trivial endomorphism of $E\times E/H$). 
In this example, $\degins(\phi)/\degins(\widehat{\phi})$ is a {\it non trivial power of $p$}, even though the heights of $E\times E'$ and $E/H_1\times E/H_2$ are {\it equal}. 
This exhibits a situation where a minimal isogeny of degree a power of $p$ preserves the differential height which is, again, in contrast with equality~\eqref{ieq:gripaz.hdiff}. (See section~\ref{Kani} for a longer discussion).   

Thirdly, Moret-Bailly \cite{MB81} manufactured an example of two {\it non ordinary} abelian surfaces $A, B$ defined over $K=\F_p(t)$, linked by an  isogeny $\phi:A\to B$ of degree $p$. 
Given the construction of $\phi$, one shows (see section~\ref{moretbabapara}) that the Frobenius isogeny $\Frob_{A/K} :A\to A^{(p)}$ factors through $\phi$:
\[ \begin{tikzcd} A \arrow[rr, "\phi"] \arrow[rrrr, "\Frob_{A/K}"', bend right=20] && B \arrow[rr, dashed, "\psi"] && A^{(p)}. \end{tikzcd}\]
The three abelian varieties in this sequence  have respective  heights 
$\hdiff(A/K) = p-1$, $\hdiff(B/K)=0$, and $\hdiff(A^{(p)}/K)=p(p-1)$.
This sequence of heights is surprising in several respects, and obliterates the hope to directly extend \eqref{ieq:gripaz.hdiff} to dimension~$2$:
the height decreases by $p-1$ through $\phi : A\to B$ (even though $\phi$ is purely inseparable of degree $p$),  
the height increases by $p(p-1)$ through the isogeny $\psi:B\to A^{(p)}$ (which has degree $p$ too). 
Finally, $B$ has height $0$ even if it is isogenous to abelian surfaces with positive height.  
We construct other, more general, examples of the same ilk in section \ref{s:moretba.steroids}.

We also observe that Theorem~\ref{itheo:hdiff} throws light on some slight discrepancies in the analogy between the arithmetic theories of abelian varieties over number fields and that over function fields (even in the simpler setting of characteristic~$0$). Compare, indeed, our equality \eqref{ieq:char0}  with the analogous result in the number field realm. There, an inequality due to Faltings 
(see \cite[Lemma 5]{Fal83}) states that if $\phi:A\to B$ is an isogeny between abelian varieties over a number field~$L$, then $\big|\hfal(A/L)-\hfal(B/L)\big| \leq \log\sqrt{\deg\phi}$, and there are examples (in \cite{SU99} for instance) where this inequality is essentially best possible.
\\

As alluded to above, the proof of \eqref{ieq:char0} -- and that of \eqref{ieq:charp.FVisog} in the case of a FV-isogeny of type $(0,0)$ -- is rather straightforward. 
Let us outline our argument for the other two assertions in Theorem~\ref{itheo:hdiff} which run in parallel and are much more delicate. We assume that $K$ has positive characteristic $p$.
By a classical d\'evissage we are reduced, thanks to part \itemref{item:hdiff.a} of Theorem \ref{itheo:hdiff}, to considering isogenies $\phi:A\to B$ of degree  a power of $p$. 
We begin by considering an isogeny $\phi$ whose kernel $G$ is contained in the kernel $A[\Frob_{A/K}]$ of the Frobenius of $A$ (\ie{}, $G$ has Frobenius height~$\leq 1$). 
Let $\Acal$ (resp. $\Bcal$) denote the N\'eron model of $A$ (resp. $B$) over $C$ and let $\Gcal$ denote the Zariski closure of $G$ in $\Acal$. Let $\Bcal'$ denote the \FPPF{} quotient $\Acal/\Gcal$.
Proposition~\ref{prop:Yuan} gives
$\hdiff(A/K)  = \deg\Omega({\Bcal'}/C) + (p-1)\,\deg\Lie(\Gcal/C)$,
where $\Omega({\Bcal'/C})$ is the Hodge bundle of $\Bcal'$.
Thanks to the {\it semi-stability} of  $A$, we can prove (Proposition \ref{prop:isomneutralcomponents}) that $\Bcal'$ and $\Bcal$ have isomorphic neutral components. In particular we have  $\deg\Omega({\Bcal'/C})=\deg\Omega({\Bcal/C}) = \hdiff(B/K)$, and this leads (Proposition \ref{prop:varheightsemistable}) to 
\begin{equation}\label{ieq:key.eq}
	\hdiff(B/K)  = \hdiff(A/K) - (p-1)\,\deg\Lie(\Gcal/C),
\end{equation}
which is the key to proving  Theorem \ref{itheo:hdiff}. 
In the special case where $\phi = \Frob_{A/K}$, Proposition \ref{prop:lie.ker.frob} asserts that $\deg\Lie(\Gcal/C) = -\hdiff(A/K)$.
We deduce that $\hdiff(B/K)=p\,\hdiff(A/K)$, and statement \itemref{item:hdiff.b} in Theorem \ref{itheo:hdiff}  then follows rather quickly.
The proof of \itemref{item:hdiff.c} requires three more steps.
If $G=\ker\phi$ is contained in $A[\Frob_{A/K}]$, a positivity result on Hodge bundles of {\it ordinary} abelian varieties gives us   sufficient control on the term $\deg\Lie(\Gcal/C)$ in \eqref{ieq:key.eq}, and we prove (Proposition~\ref{prop:keyinequalityordinary}) that
\[\hdiff(A/K)\leq \hdiff(B/K)\leq p\,\hdiff(A/K),\]
which is inequality \eqref{ieq:charp.gen} for an isogeny $\phi$ such that $\ker\phi$ has Frobenius height $\leq1$.
Secondly, we  proceed by induction on the Frobenius height of $\ker\phi$ to obtain  \eqref{ieq:charp.gen} for an isogeny with connected kernel.
Thirdly, a duality argument allows to prove \eqref{ieq:charp.gen} if $\ker\phi$ is \'etale from the `connected kernel' case. 
There remains to invoke the decomposition of an isogeny in separable and purely inseparable parts to obtain \eqref{ieq:charp.gen} in general.

A few of the technical steps summarized above gave rise to questions, examples and counter-examples, some of which we also discuss in the course of the text when relevant. 
In various places, we revisit and rewrite some of the proofs in the literature (for instance Proposition~\ref{prop:thmprincipalYuan}) in an effort to improve readability and to ensure that this paper is more self-contained.
\\

We also generalize \eqref{ieq:gripaz.hmod} to higher-dimensional abelian varieties.
 In dimension $g>1$, the modular height as defined in Definition~\ref{defi:hmod} is a natural generalization of the Weil height of the elliptic $j$-invariant. 
 By virtue of Moret-Bailly's Formule Cl\'e (see Theorem \ref{theo-clef}), we can transfer the control given in Theorem~\ref{itheo:hdiff} on the variation of the differential height through an isogeny, into a relation between modular heights. 
 This is the object of the following statement.

	\begin{icoro}\label{itheo:hmod}
	Let $\phi : A\to B$ be an isogeny between  semi-stable abelian varieties defined over $K$. 
	Let $\xi$ (resp. $\xi'$) be a polarization on $A$ (resp. on $B$) of degree $\deg\xi$ (resp. $\deg\xi'$), and $\nu$ (resp. $\nu'$) be a level-$n$ structure on $A$ (resp. on $B$). 
	If $K$ has positive characteristic $p$, assume that $A$ and $B$ are ordinary.
	Then
	\[p^{-\delta_p(\widehat{\phi})}\,\frac{\hmod(A, \xi, \nu)}{\deg\xi} 
	\leq \frac{\hmod(B, \xi', \nu')}{\deg \xi'} 
	\leq p^{\delta_p(\phi)}\,\frac{\hmod(A, \xi, \nu)}{\deg\xi},\]
	where $\delta_p(-)$ denotes the Frobenius height of an isogeny (to be interpreted as $0$ if $K$ has characteristic $0$).
\end{icoro}
\bigskip

The second principal result of this paper is a perfect analogue, in the function field setting, to the parallelogram inequality of R\'emond for abelian varieties over a number field (see \cite{R22}). In section~\ref{parallelogram}, we indeed prove
\begin{itheo}\label{itheo:parineq}
  Let $K$ be the function field of a curve as above.
  Let $A$ be an abelian variety over $K$. If $K$ has positive characteristic, assume that $A$ is semi-stable.
  For any   finite subgroup schemes $G,H$ of $A$, we have
  \begin{equation}
      \label{ieq:parineq}
      \hdiff\big((A/(G+H))/K\big) +\hdiff\big((A/(G\cap H))/K\big) \leq \hdiff\big((A/G)/K\big)  +\hdiff\big((A/H)/K\big).
  \end{equation}
  Moreover, if at least one of the orders of $G$ or $H$ is coprime to the characteristic of $K$
  (e.g. if the characteristic of $K$ is $0$), 
  then~\eqref{ieq:parineq} is an equality.
\end{itheo}

 The fact that such an inequality holds for abelian varieties over $K$ is {\it a priori} not obvious: in view of Theorem~\ref{itheo:hdiff}, the differential height of the quotient of $A$ by a finite subgroup scheme $G$ may be radically different from $\hdiff(A/K)$, and, in the absence of constraints on $A$ or $G$, we may not know exactly how $\hdiff((A/G)/K)$ compares to $\hdiff(A/K)$.
In that respect, it is remarkable that Theorem \ref{itheo:parineq} holds without an ordinarity condition on $A$, or without any assumption on the subgroup schemes.
Our parallelogram inequality (for the differential height of abelian varieties over function fields) looks exactly the same as R\'emond's (which concerns the Faltings height of abelian varieties over number fields), and our inequality does not involve the Frobenius heights of $G,H$. 
This is surprising too because Theorem~\ref{itheo:hdiff} above shows that the differential height varies differently through an isogeny compared to the Faltings height. 
  
 Our proof of Theorem~\ref{itheo:parineq} is different from R\'emond's proof  in \cite{R22}: the latter is `local' in nature, while ours follows a `global' approach.
 We define and study parallelogram configurations, and we first prove three special cases of Theorem~\ref{itheo:parineq}: namely, when $G$ or $H$ is \'etale with \'etale dual, when $G$ or $H$ is connected, and when $G$ or $H$ is \'etale with order a power of $p$. 
 The final step is a d\'evissage based on a well-chosen filtration of one of the subgroup schemes, which allows to conclude  thanks to the key Lemma~\ref{lemm:parineq.devissage}.

A striking consequence of the parallelogram inequality \eqref{ieq:parineq} is the fact that the stable differential height is superadditive (in arbitrary characteristic) on  short exact sequences  of abelian varieties
$0\to B\to A\to A/B\to 0$, where $B\to A$ is the inclusion (the stable Faltings height of abelian varieties over number fields also enjoys this property by \cite{R22}). 
In other words, we have the following statement, which we prove with the same method as Rémond.
	\setcounter{corocount}{1}
    \begin{icoro}\label{itheo:sub}
    Let $A$ be an abelian variety over $K$. If $K$ has positive characteristic, assume that $A$ is semi-stable. For any abelian subvariety $B$ of $A$, we have
    \begin{equation}\label{ieq:ineq.hsub}
        \hdiff(B/K) + \hdiff((A/B)/K) \leq \hdiff(A/K).
    \end{equation} 
    Moreover, if $K$ has characteristic $0$, inequality (\ref{ieq:ineq.hsub}) is an equality.
    \end{icoro}
  
For an abelian subvariety $B$ of $A$, the Poincar\'e reducibility theorem yields a quasi-complement $B'$ of $B$ in $A$ {\it i.e.}, an abelian subvariety $B'\subset A$ such that 
the sum map $\sigma: B\times B'\longrightarrow A$ is an isogeny.
Since $A/B$ is isogenous to $B'$, our Theorem~\ref{itheo:hdiff} suggests that the heights of $B$ and $A/B$ should somehow be related to $\hdiff(A/K)$.  
It strikes us as unexpected, though, that Corollary~\ref{itheo:sub} provides such a clean inequality between these heights: inequality \eqref{ieq:ineq.hsub} does not involve any terms depending on the isogeny $\sigma$ and requires no ordinarity assumption. 
Note that inequality \eqref{ieq:ineq.hsub} may be strict (see Example~\ref{exple:strict.ineq.subvar}), and that
Corollary \ref{itheo:sub} is actually equivalent to Theorem \ref{itheo:parineq}, see Remark~\ref{rema:equiv.parineq.sub}.

When the base field is a number field, the analogue of Theorem~\ref{itheo:parineq}, or more precisely of its Corollary~\ref{itheo:sub}, is one of the technological jumps which made the impressive improvements in the explicit bounds from \cite{GR14} to \cite{GR24} possible. 
An inequality such as \eqref{ieq:ineq.hsub} is indeed a perfect tool for a proof by induction on the dimension of the abelian variety since there are no error terms and no dependency on the dimension or the codimension. 
Similar Diophantine applications in the case of abelian varieties over a function field are currently under development. 

In the case of elliptic curves over $K$, Theorem~\ref{itheo:parineq}  yields the following inequality between the absolute heights of elliptic $j$-invariants. 
	\setcounter{corocount}{2}
	\begin{icoro}\label{itheo:parineq.ellcurves}
	Let $E$ be a elliptic curve over $\bar{K}$, 
	and let $G,H$ be two finite subgroup schemes of $E$. Then we have
    \[ h\big(j(E/(G+H))\big) + h\big(j(E/(G\cap H))\big)\leq h\big(j(E/G)\big) + h\big(j(E/H)\big). \]
	\end{icoro}

Corollary~\ref{itheo:parineq.ellcurves} could be applied to explicit situations derived from Richelot isogenies (see \cite{FS22}, for instance).

\paragraph{Plan of the paper --}  
In the first two sections, we gather the necessary background on heights of abelian varieties and on isogenies, respectively. 
We prove equality \eqref{ieq:char0} of Theorem~\ref{itheo:hdiff} there.
Section~\ref{sec:FV} focuses on the case of a positive characteristic ground field: we introduce the Frobenius and Verschiebung isogenies there, and recall results on Hodge bundles and group schemes of Frobenius height $1$. 
In Section~\ref{sec:Hodge and heights}, we study Hodge bundles in more depth, and relate N\'eron models of isogenous abelian varieties.
We also describe the behaviour of the differential height for special types of isogenies (Frobenius, Verschiebung, and FV-isogenies) \ie{}, we prove part \itemref{item:hdiff.b} of Theorem \ref{itheo:hdiff}. 
In Section~\ref{sec:proof of theorem}, we conclude the proof of the remaining assertion of Theorem~\ref{itheo:hdiff}, using the tools developed earlier, as well as a positivity result for certain Hodge bundles. 
In Section~\ref{sec:ex and cex} we provide examples and counterexamples which illustrate some subtle properties of general isogenies in characteristic $p>0$. 
Finally, we prove Theorem~\ref{itheo:parineq} and its corollaries in Section~\ref{parallelogram}.

\paragraph{General notation --}  
 In the rest of the paper, $S$
 denotes an integral locally noetherian scheme of absolute dimension at most $1$. In the applications we have in mind, $S$ is either the spectrum of a field, or a smooth projective curve over a field.
We also let $k$ be a perfect field, and $C$ be a smooth projective geometrically irreducible curve over $k$. We write $K = k(C)$ for the function field of $C$.


\numberwithin{equation}{section}
\section{Preliminaries on heights} \label{sec:prel.heights}

\subsection{Hodge bundles and Lie algebras}
Let $\Gcal$ be a group scheme over $S$  with neutral section $e:S\to\Gcal$. We let 
\[\Omega(\Gcal/S) := e^\ast\,\Omega^1_{\Gcal/S}\]
denote the pullback by $e$ of the sheaf of relative differential $1$-forms on $\Gcal/S$.
This sheaf on $S$ is called {\it the Hodge bundle of $\Gcal/S$}.
In the same setting, the group functor $\underline{\operatorname{Lie}}(\Gcal/S)$ defined by 
	\[T \mapsto \ker \big(\Gcal(T[\varepsilon]) \rightarrow \Gcal(T)\big),\]
	where $T[\varepsilon] : = T \times \Spec (\Z[X]/X^2)$, is represented by a $S$-group scheme denoted by $\operatorname{\mathbb{L}ie}(\Gcal/S)$.
We  denote by $\Lie(\Gcal/S)$ the $\Ocal_S$-module of sections of  $\operatorname{\mathbb{L}ie}(\Gcal/S)$ viewed as a vector fibration 	(see \cite[Exposé II, Notation 4.11.6]{SGA3}, where the same object is denoted by $\mathscr{L}ie (\Gcal/S)$).

	\begin{prop}\label{prop:duality.hodge.lie}
	If $\Omega(\Gcal/S)$ is locally free of finite rank, there is an isomorphism of $\O_S$-modules $\operatorname{Lie}(\Gcal/S) \simeq \Omega(\Gcal/S)^\vee$.
	\end{prop}
	\begin{proof}
	This is \cite[Exposé II, Lemme 4.11.7]{SGA3}.
	\end{proof}
	
\subsection{Differential and stable heights}

	Let $A$ be an abelian variety of dimension $g$ defined over $K$.
	Let $\Acal \to C$ denote the N\'eron model of $A/K$, and write $e:C\to\Acal$ for its neutral section.
	We write
	\[	\Omega_A := \Omega(\Acal/C)	\] 
	for the Hodge bundle of $\Acal/C$ when no confusion can arise. 
	Notice that $\Omega(\Acal/S)$ is dual to $\Lie(\Acal/S)$ by Proposition~\ref{prop:duality.hodge.lie} (the sheaf $\Omega^1_{\Acal/C}$  is indeed a locally free $\Ocal_{\Acal}$-module of rank $g$ by \cite[Proposition 2.5 p. 222]{Liu02}).
	
	Recall that the degree of a locally free $\O_C$-module of finite rank is the degree of its maximal exterior power.
		
	\begin{defi}\label{defi:hdiff}
	The {\it differential height} of an abelian variety $A/K$ of dimension $g$ is the degree of the Hodge bundle of its N\'eron model: 
	\[\hdiff(A/K) := \deg(\Omega_A) = - \deg(\Lie (\Acal/C))\]
	\end{defi}

This notion of height has first been introduced by Parshin \cite{Parshin} in his work on the Mordell conjecture over function fields.
Proposition 2.3 in \cite[Chapitre XI]{MB85}  shows that, for any abelian variety $A/K$ and any finite extension $K'/K$, we have
\[\hdiff((A\times_{K} K')/K')\leq [K':K]\,\hdiff(A/K).\]
The proof of this inequality uses the following lemma, which describes the variation of the degree through certain morphisms. 
Such a result is also relevant in the theory of the Harder--Narasimhan filtrations for instance. 
We include a proof because this lemma is needed later in this paper, and we could not find a satisfactory reference.

	\begin{lemm}\label{lem:keyinequalityparallelogram}
    The following two properties hold.
	\vspace{-0.6\topsep}\begin{enumerate}[{\rm(}a{\rm)}]
	\setlength{\itemsep}{0em}
    \item For any line bundles $\mathcal{L}$ and $\Lcal'$ on $C$, 
	if there is a map of line bundles $\Lcal \rightarrow \Lcal'$ whose generic fiber is an isomorphism, then $\deg \Lcal \leq \deg \Lcal'$.
    \item For any vector bundles $\Ecal$ and $\Ecal'$ on $C$, 
	if there is a map of vector bundles $\Ecal \rightarrow \Ecal'$ which is an isomorphism on the generic fiber, then $\deg \Ecal \leq \deg \Ecal'$.
	\end{enumerate}    
	\end{lemm}
	\begin{proof}
    $(a)$ If there is such a map $\Lcal\to\Lcal'$, tensoring by $\Lcal^{-1}$ yields a morphism of line bundles $\Ocal_C \rightarrow \Lcal^{-1} \otimes \Lcal'$ which is an isomorphism on the generic fiber. 
    Considering the image by that map of the canonical trivialising section of $\Ocal_C$ 
    and identifying $\Lcal^{-1} \otimes \Lcal'$ with some $\Ocal_C(D)$ for a divisor $D$ on $C$, 
    we thus obtain a non-zero global section of $\Ocal_C(-D)$, which in turn corresponds to a nonzero meromorphic function $f$ on $C$ such that  $\div(f) \geq -D$. 
    This implies that $\deg(D)  \geq 0$, which means that $\deg (\Lcal^{-1} \otimes \Lcal') \geq 0$. In other words, we have  $\deg \Lcal \leq \deg \Lcal'$.

    $(b)$ Let $r$ be the common rank of $\Ecal$ and $\Ecal'$. The degrees of the line bundles $\wedge^r \Ecal$ and $\wedge^r \Ecal'$ are the degrees of $\Ecal$ and $\Ecal'$ by definition. 
    The statement directly reduces to $(a)$ because the map $\Ecal \rightarrow \Ecal'$ induces a map of line bundles $\wedge^r \Ecal \rightarrow \wedge^r \Ecal'$ which is an isomorphism on the generic fiber. 
    \end{proof}

If $K'$ is a finite extension of $K$, and if $A$ is an abelian variety over $K'$, the Semi-stable Reduction Theorem (see \cite{Gro72}) asserts that there exists a finite extension $K_{s}/K'$ such that the base change $A\times_{K'}K_{s}$ has semi-stable reduction everywhere (see Definition~\ref{defi:semistable} below for a reminder of the definition).The ratio ${\hdiff((A\times_{K'} K_{s})/K_{s})/[K_{s}:K]}$ is independent of the choice of extension $K_{s}/K'$ over which $A$ attains semi-stable reduction, hence it is licit to define the {\it stable height} of $A$ by 
    \[ \hstab(A) := \frac{ \hdiff((A\times_{K'} K_{s})/K_{s})}{[K_{s}:K]}   \]
    (for more details see \cite[Chapitre XI, Définition 2.4.2]{MB85}).

We recall the following classical property of the differential height:
	\begin{prop}\label{prop:hdiff.product}
    Let $A$ and $B$ be abelian varieties over $K$. 
    Then we have $\hdiff((A\times B)/K) = \hdiff(A/K)+\hdiff(B/K)$. 
	\end{prop}
	\begin{proof}
 	This  is a direct consequence of the following lemma
 	since $\Omega_{A \times B} \simeq \Omega_A \times \Omega_B$.
	\end{proof}

	\begin{lemm}\label{lemm:suiteexactefibresvect}
	Assume  that we have an exact sequence 
	$0 \longrightarrow \Fcal_1 \longrightarrow \Fcal \longrightarrow \Fcal_2 \longrightarrow 0$ of locally free vector bundles of finite rank over $S$.
	Then the two $\O_S$-modules of rank $1$, $\wedge^{\mathrm{rk}(\Fcal_1) + \mathrm{rk}(\Fcal_2)}(\Fcal)$ and $\big( \wedge^{\mathrm{rk}(\Fcal_1)}\Fcal_1  \big)\otimes \big(\wedge^{\mathrm{rk}(\Fcal_2)}\Fcal_2 \big)$  are canonically isomorphic. 

	In particular, if $S$ is a smooth projective curve over $k$, we have $\deg(\Fcal) = \deg(\Fcal_1) + \deg(\Fcal_2)$.
	\end{lemm}
	\begin{proof}
	For $i=1, 2$ we write $r_i = \mathrm{rk}(\Fcal_i)$.
    We first show that there is a canonical map 
    $g:\wedge^{r_1}\Fcal_1  \otimes \wedge^{r_2}\Fcal_2 \longrightarrow\wedge^{r_1+r_2}\mathcal{F}$ of $\O_S$-modules.
	By assumption there is an alternating map $f :\Fcal_1^{r_1}\times \Fcal^{r_2} \to \wedge^{r_1+r_2}\Fcal$.
	Since $\wedge^{r_1}\Fcal_1$ is the maximal exterior power of $\Fcal_1$, this map factors as an alternating map $\Fcal_1^{r_1}\times \Fcal_2^{r_2} \to \wedge^{r_1+r_2}\Fcal$. 
	By the universal properties of the tensor product and of the exterior power, the latter map further induces a canonical map $g$ as desired.
	Here is a diagram summarizing the situation where the two  vertical arrows denote $\O_S$-linear maps.  
	\begin{center}
	\begin{tikzcd}
	\Fcal_1^{r_1}\times\Fcal^{r_2} \arrow[d] \arrow[rrd, "f"]  &      &                       \\
	\Fcal_1^{r_1}\times\Fcal_2^{r_2} \arrow[rrd] \arrow[rr, dashed] &                       & \wedge^{r_1+r_2}\Fcal \\
	& & (\wedge^{r_1}\Fcal_1)\otimes(\wedge^{r_2}\Fcal_2) \arrow[u, "g"', dashed] 
	\end{tikzcd}
	\end{center}
    There  remains to prove that $g$ is an isomorphism: this can be checked locally at all points of $S$. The corresponding statement in commutative algebra is  \cite[Chapter XIX, Proposition 1.2]{LangAlgebra}.
	\end{proof}

\subsection{Modular height}\label{modular setting}
Let $g,d\geq 1$, and $n\geq 3$ be  integers.
Let $\mathscr{A}_{g,d,n}$ be the (fine) moduli space parametrizing triples $(A,\xi,\nu)$ of abelian varieties $A$ of dimension $g$ equipped with a polarization $\xi$  of degree $d^2$, and a level structure~$\nu$ of level $n$ (see for instance \cite[p. 176]{MB85}). 
Let $\overline{\mathscr{A}}_{g,d,n}$ be a suitable compactification of $\mathscr{A}_{g,d,n}$. We equip $\overline{\mathscr{A}}_{g,d,n}$ with an ample line bundle $\mathscr{L}$, as in \cite[p. 178]{MB85}: this allows to define a height function  $h_{\mathscr{L}}$ on $\overline{\mathscr{A}_{g,d,n}}(\overline{K})$.
    
For any finite extension $K'$ of $K$, a triple $(A,\xi,\nu)$ over $K'$ induces a $K'$-rational point on ${\overline{\mathscr{A}}_{g,d,n}}$ {\it i.e.}, a morphism  
    \[j(A,\xi,\nu): \Spec(K')\longrightarrow \overline{\mathscr{A}}_{g,d,n},\]
as given in (3.4.3.2) of \cite[p. 178]{MB85}. 

	\begin{rema}
	The existence of a level structure of level $n$ with $n\geq 3$ on an abelian variety $A'/K'$ implies that $A'$ is semi-stable over $K'$, as remarked in \cite[p. 236]{MB85}.
	\end{rema}
    
	\begin{defi}\label{defi:hmod}
	Let $(A,\xi,\nu)$ be a triple over $K'$ as above. The \textit{modular height} of $(A,\xi,\nu)$ is defined as 
    \[\hmod(A, \xi,\nu):=h_{\mathscr{L}}\big(j(A,\xi,\nu)\big).\]
	\end{defi}
	
	\begin{lemm}\label{modpoz}
	In the above context, we have $\hmod(A, \xi, \nu)\geq0$. 
		
	Moreover, $\hmod(A, \xi, \nu)=0$ if and only if $A$ is constant i.e., there exists  an abelian variety $A_0$ over $k$ such that $A_0\times_{k}{K'}$ is isomorphic to $A$.
	\end{lemm}
	\begin{proof}
	This is Propri\'et\'e 3.2.3 on p.\,218 of \cite{MB85}, see also Th\'eor\`eme 4.5 on p.\,235 of \cite{MB85}. 
	\end{proof}

\subsection{Comparison of heights}
The best way to compare these  notions of height -- from Definition~\ref{defi:hdiff} and Definition~\ref{defi:hmod} -- is to follow Moret-Bailly, and to study his {`Formule Cl\'e'}. We only state his key formula  in our particular case, keeping the notation from the previous paragraph.
    
	\begin{theo}[Moret-Bailly's Formule Cl\'e]\label{theo-clef}
	In the above context, for any point $j(A,\xi,\nu)\in{\overline{\mathscr{A}}_{g,d,n}}(K')$, we have 
		\[\hmod(A,\xi,\nu)=\frac{4^g d}{2}\, \hstab(A).\]
	\end{theo}
	\begin{proof}
	This is \cite{MB85} Corollaire 3.1.1 on p.\,232, see also \cite[XI.3.2]{MB85} Th\'eor\`eme 3.2 on pp.\,233--234.
	\end{proof}
	
	\begin{coro}\label{cor:MoretBailly}
	For any abelian variety $A$ over $K$, we have $\hdiff(A/K) \geq 0$.  
		
    Moreover, for any finite extension $K'/K$ such that there exists a triple $(A,\xi,\nu)$ over $K'$, 
    we have $\hstab(A)=0$ if and only if the abelian variety $A\times_K{K'}$ over $K'$ is constant. 
	\end{coro}
	
	\begin{proof}
	The result follows by combining the inequality $\hdiff(A/K)\geq \hstab(A)$, Theorem~\ref{theo-clef}, and Lemma~\ref{modpoz}.
	\end{proof}

\section{Preliminary results on isogenies}

The proof of Theorem~\ref{itheo:hdiff} is much simpler in the case of isogenies of degree coprime to the characteristic of the base field $K$ (equality \eqref{ieq:char0}, and equality \eqref{ieq:charp.FVisog} for FV-isogenies of type $(0,0)$). 
We therefore treat this case  at the end of this section, in which
we first introduce the main characters of this paper, {\it viz.} isogenies. 

\subsection{Isogenies} 
We begin by recalling a few basic facts about isogenies, as well as some of their proofs, because we found them to be somewhat scattered through the literature or difficult to locate. 
Our presentation is heavily inspired by the (as of yet unpublished) book \cite{EGM_AV} by Edixhoven, van der Geer, and Moonen. The reader is also referred to \cite[\S 7.3]{BLR}. 
 
	\begin{defi}
	A homomorphism $\phi: G \rightarrow H$ of connected group varieties over a field $K$ is called an \emph{isogeny} if it is surjective (\ie{}, the underlying map on sets is), and its kernel is a finite group scheme over $K$. 
  	The order of the kernel of $\phi$ (viewed as a group scheme over $K$) is then called the \emph{degree of $\phi$}.
		
	An isogeny $\phi$ is called {\it separable} if it is an \'etale  morphism of algebraic varieties, and {\it purely inseparable} if it is a purely inseparable (also called radicial or universally injective) morphism of algebraic varieties (see \cite[\S 3.5]{EGAI}).
	\end{defi}

Here are well-known criteria for a homomorphism to be an isogeny:

	\begin{prop}
	Let $\phi: A \rightarrow B $ be a homomorphism  between abelian varieties over a field $K$. 
	The following statements are equivalent:
	\vspace{-0.6\topsep}\begin{enumerate}[{\rm(}i{\rm)}]
	\setlength{\itemsep}{0em}
    	\item $\phi$ is an isogeny, 
	    \item $\phi$ has finite kernel and $\dim(A) = \dim(B)$,
    	\item $\phi$ is finite, flat, and surjective.
	\end{enumerate}
	\end{prop}
	\begin{proof}
	We prove $(i) \Leftrightarrow (ii)$ first.
	We have $\dim \phi(A) = \dim A - \dim \ker \phi$, so that $\dim B = \dim \phi(A) = \dim A$ if $\phi$ is an isogeny. Conversely, if $\ker \phi$ is finite  and $\dim A = \dim B$, then $\dim B = \dim \phi(A)$ and $\phi(A)$ is a proper subvariety of $B$. In that case, $\phi(A) = B$ since $B$ is connected, and $\phi$ is thus surjective.
	We prove $(ii) \Leftrightarrow (iii)$ next. Assuming $(ii)$,  the same argument as above proves that $\phi$ is surjective. The fibers of $\phi$ are all translates of $\ker(\phi)$, so the sheaf $\phi_* \Ocal_A$ is a locally free $\Ocal_B$-module of finite rank $r$. 
	This rank $r$ equals $[K(A):K(B)]$ at the generic point of $B$, while $r$ is the order of $\ker \phi$ at $0_B$ by definition. If $\phi$ has finite kernel, $K(A)$ and $K(B)$ have the same transcendence degree over $K$, hence $\dim A = \dim B$, and $\phi$ is finite and flat \cite[Chapter IV, Remark 3.11]{Liu02}.  The converse is immediate.
	\end{proof} 

    \begin{prop}\label{prop:basics.isog}
	Let $\phi: A \rightarrow B $ be an  isogeny  between abelian varieties over a field $K$.	
        	\vspace{-0.6\topsep}\begin{enumerate}[{\rm(}a{\rm)}]
        	\setlength{\itemsep}{0em}
            \item The degree of $\phi$ equals the degree of the extension of function fields $K(A)/K(B)$ induced by $\phi^*$.
                    \item  The isogeny $\phi$ is separable (resp. purely inseparable) if and only if its kernel is an \'etale (resp. connected) $K$-group scheme.
            \item The isogeny $\phi$ is separable (resp. purely inseparable) if and only if the field extension $K(A)/K(B)$ is.
        \end{enumerate}

    \end{prop}

In case $K$ has characteristic $0$, this result implies that any isogeny between abelian varieties is separable.

\begin{proof}
	$(a)$ has already been proved in the proof of the previous proposition. We now prove $(b)$ and $(c)$.    
    First, $\phi$ is \'etale if and only if $\ker \phi$ is \'etale by equality of dimensions and \cite[Proposition 1.63]{Milne}. Moreover, if it is \'etale, 
	the maps between residue fields must always be \'etale (i.e. finite separable), so at the generic points we obtain that $K(A)/K(B)$ is separable. Conversely, assume $K(A)/K(B)$ is (finite) separable. This implies that for some nonempty open $U$ of $B$, $\phi_{|\phi^{-1} (U)} : \phi^{-1}(U) \rightarrow U$ is \'etale (see e.g. \cite[Chapter VI, Exercise 2.9]{Liu02}), but $U$ can then be translated (and by all $B(K\usep))$ as $\phi$ is a morphism so $\phi_{K\usep}$ is \'etale everywhere, and by separable descent $\phi$ itself is \'etale.
	
	Now, if $\phi$ is purely inseparable, by \cite[Proposition 3.7.1]{EGAI}, the induced maps between residue fields are purely inseparable everywhere, in particular $K(A)/K(B)$ is purely inseparable. Conversely, if $K(A)/K(B)$ is purely inseparable, by spreading out \cite[Theorem 8.10.5]{EGAIV3}, there is a nonempty open subset $U$ of $B$ for which $\phi$ restricted to $\phi^{-1}(U)$ towards $U$ is purely inseparable, and again this extends to all $B$ as $\phi$ is a morphism of algebraic groups.

To finish, let us prove why $\phi$ is purely inseparable if and only if $\ker \phi$ is a connected group scheme. If $\phi$ is purely inseparable, it is universally injective, therefore for any field $L/K$, $\phi : A(L) \rightarrow B(L)$ is injective, so $(\ker \phi) (L) = \{0_A\}$. In particular, $\ker \phi$ has only one (closed) point which is $0_A$.

 As $\ker \phi$ is finite flat, define $R$ the associated $K$-algebra. It is of finite $K$-dimension hence artinian. 
 
 Having assumed $\ker \phi$ has only one closed point, means that $R$ is also local so it has only one prime ideal (as in artinian rings all prime ideals are maximal). In particular, $R$ is irreducible so $\ker \phi$ is connected. Conversely, assume $\ker \phi$ is connected. For every field extension $L/K$ it stays connected by invariance by scalar extension, so by similar arguments, $R \otimes_K L$ is always local, therefore by the exact sequence $0 \rightarrow (\ker \phi)(L) \rightarrow A(L) \rightarrow B(L)$, we have that $\phi_L : A(L) \rightarrow B(L)$ is always injective, hence $\phi$ is purely inseparable as it is universally injective.
\end{proof}

We also recall the notion of dual isogeny and its relation with the Cartier dual of a group scheme:
	\begin{prop}\label{prop:dualisogenykernel}
	Let $\phi : A \rightarrow B$ be an isogeny of abelian varieties over $K$. Let $A^\vee := \Pic^0(A)$ (resp. $B^\vee:=\Pic^0(B)$) denote the dual abelian variety of $A$ (resp. $B$).
		
	Then, the pullback map $\phi^* : \Pic^0(B) \rightarrow \Pic^0(A)$ induced by $\phi$ is an isogeny of abelian varieties, denoted by $\widehat{\phi} : B^\vee \rightarrow A^\vee$, and called the {\it dual of $\phi$}. 
	
	Moreover, the kernel $\ker \widehat{\phi} $ is isomorphic (as a group scheme over $K$) to the Cartier dual of $\ker \phi$.
	\end{prop}
	\begin{proof}
	This is \cite[\S 15, Theorem 1]{MumfordAbVar08}.
	\end{proof}

\subsection{Semi-stability and extensions of isogenies to Néron models}
We now specialize to the case where the base field $K$ is the function field of a curve $C$ over a perfect field $k$. 
We first recall the notion of semi-stable abelian variety.

	\begin{defi}\label{defi:semistable}
    Let $A$ be an abelian variety over $K$, and let $\Acal$ denote its Néron model over $C$.
    We say that $A$ is \emph{semi-stable} (over $K$) if, for every closed point $s \in C$, 
    the fiber $\Acal_s$ is semi-abelian {\it i.e.}, its  identity component $(\Acal_s)^\circ$ is the extension of an abelian variety by an affine torus.
	\end{defi}

	\begin{prop}\label{prop:semistableisogenies}
	Let $\phi : A\to B$ be an isogeny between abelian varieties over $K$, with respective N\'eron models $\Acal$ and $\Bcal$ over $C$. Let $\Phi : \Acal \rightarrow \Bcal$ denote the extension of $\phi$ to N\'eron models (by the Néron mapping property). 
	Then
	\vspace{-0.6\topsep}\begin{enumerate}[{\rm(}a{\rm)}]\setlength{\itemsep}{0em}
		\item $A$ is semi-stable if and only if $B$ is semi-stable.
		\item Assume that the characteristic of $K$ is coprime to $\deg \phi$ or that $A$ is semi-stable. Then,  for any closed point $s\in C$, the restriction ${\Phi_s : \Acal_s\to \Bcal_s}$ of $\Phi$ to the fibers over $s$ is finite and surjective on identity components \ie{},   $\Phi_s^\circ: \Acal_s^\circ \rightarrow \Bcal_s^\circ$ is an isogeny.
		\item Assuming that the characteristic of $K$ is coprime to $\deg \phi$, the isogeny $\Phi_s^\circ: \Acal_s^\circ \rightarrow \Bcal_s^\circ$ mentioned in the previous item is \'etale for any closed point $s\in C$.
	\end{enumerate}
	\end{prop}
	\begin{proof}
    These assertions directly follow from \cite[\S VII.3, Corollary 7]{BLR},  \cite[\S VII.3, Proposition 6]{BLR}, and  \cite[\S VII.3, Lemma 2]{BLR} (or \cite[\S1.1, Exemples 1.1.2]{Ray}), respectively.
	\end{proof}
 
\subsection{Decompositions of isogenies}
In order to prove our main theorem, we will need to perform various `d\'evissages' based, in particular, on the fact that an isogeny factors into separable and inseparable parts.

We begin by recalling the following:
	\begin{defi}\label{defi:exact.sequence.groupschemes}
	A complex 
 	$0 \longrightarrow \Gcal_1 \xrightarrow[]{\ i \ }   \Gcal_2 \xrightarrow[]{\ \pi \ }  \Gcal_3$ 
	of group schemes over $S$ is a \emph{left exact sequence} if $i$ is a closed immersion and if, for every morphism $T \rightarrow S$, the sequence
	\[0 \longrightarrow \Gcal_1 (T) \longrightarrow \Gcal_2 (T) \longrightarrow \Gcal_3 (T)\]
	is exact. 
	Assuming that  the three involved group schemes are flat of finite type over $S$ (hence \FPPF{}), the complex 
 	\[0 \longrightarrow \Gcal_1 \xrightarrow[]{\ i \ } \Gcal_2 \xrightarrow[]{\ \pi \ } \Gcal_3 \longrightarrow 0 \]
 	is an \emph{exact sequence} if it is left exact and if, moreover, $\pi$ is faithfully flat. 
 	The condition for the aboove complex to be an exact sequence is equivalent to asking that $\Gcal_3$ represent  the \FPPF{} quotient $\Gcal_2/i(\Gcal_1)$ (see \cite[Definition 5.1.18]{Po} for more details).
	\end{defi}

	\begin{prop}\label{prop:lissitequotientsfppf}
    For every  flat, locally of finite type group scheme $\Gcal$ over $S$ and every flat subgroup scheme $\Hcal$ of~$\Gcal$, the \FPPF{} quotient $\Gcal/\Hcal$ is represented by a scheme (over $S$), which is a group scheme over $S$ when $\Hcal$ is normal in $\Gcal$.

    Furthermore, if $\Gcal$ is smooth, $\Gcal/\Hcal$ is also smooth and the quotient morphism $\Gcal \rightarrow \Gcal/\Hcal$ is faithfully flat.
	\end{prop}
	\begin{proof}
    The first part is \cite[Theorem 4.C]{Anantharaman} since we have assumed that $S$ is locally noetherian of dimension at most~$1$. The second assertion is \cite[Theorem 2.8]{Yua}.
	\end{proof} 
	
	The following statement is the so-called `connected-\'etale decomposition' of a group scheme:
	\begin{lemm}\label{lemm:connectedetalesequence}
	Let $G$ be a finite flat group scheme over an henselian local ring $R$. 
    Let~$G^\circ$ be the clopen group scheme of $G$ containing the connected component of the neutral section. 
    Then $G^\circ$ is finite and flat over~$R$, and the quotient scheme $G^{\textrm{\'et}} := G/G^\circ$ is a finite \'etale group scheme over $R$. In other words, there is an exact sequence of finite flat group schemes over $R$:
	\[0 \longrightarrow G^\circ \longrightarrow G \longrightarrow G^{\textrm{\'et}} \longrightarrow 0,\]
	where $G^\circ$  is connected,  flat over $R$, and is normal in $G$,   and $G^{\textrm{\'et}}$ is \'etale over $R$.
	
	Furthermore, this decomposition of $G$ is unique up to isomorphism.
	\end{lemm}
	\begin{proof}
	This is \cite[Chapter V, \S 3.7]{CornellSilvermanMF}, see also \cite[\S3, Proposition  p. 43]{CornellSilverman}.
	Uniqueness follows from the fact that a group scheme which is both connected and \'etale is trivial.
	\end{proof}
		
We are finally in a position to state two decomposition statements (see \cite[Corollary 5.8]{EGM_AV}) that we will require in our study.
	\begin{prop}\label{prop:isogeny.factor.sep.insep}
	An isogeny $\phi:A\to B$ between abelian varieties over $K$ factors as $\phi =\psi\circ\eta$, where $\eta:A\to A'$ is a purely inseparable isogeny, and $\psi : A'\to B$ is a separable isogeny.
	This factorization is unique up to isomorphism. 
  
	We call $\eta$ (resp. $\psi$) the purely inseparable (resp. separable) part of $\phi$ and we denote it by  $\phi\ins$ (resp. $\phi\sep$).
	We write $\deg\ins(\phi) = \deg(\phi\ins)$ for the inseparability degree of $\phi$, and $\deg\sep(\phi) = \deg(\phi\sep)$ for its separability degree. We then have $|(\ker\phi)(\overline{K})| = \deg\sep(\phi)$.  
	\end{prop}
Note that Proposition 4.7 in \cite{GriPaz} provides a more precise decomposition of an isogeny in the $1$-dimensional case.
	\begin{proof}
	For the existence part of the statement, consider the subgroup scheme $G = \ker \phi$ of $A$, and its neutral component~$G^\circ$. 
    Let  $\eta: A \rightarrow A/G^\circ$ denote the quotient isogeny. 
    Since $G^\circ\subset G$, $\phi$ factors through $\eta$, and we obtain an isogeny $\psi: A/G^\circ \rightarrow B$.
    The kernel of $\eta$ is $G/G^\circ$, hence it is an \'etale group scheme by Lemma~\ref{lemm:connectedetalesequence}. 
    Proposition~\ref{prop:basics.isog} proves that $\eta$ is purely inseparable, and that $\psi$ is separable (their respective kernels are connected, and \'etale, respectively). 

	Finally, because $\eta$ is purely inseparable, it must be injective on $A(\overline{K})$  (being purely inseparable is equivalent to being universally injective). 
	In particular, $(\ker \phi)(\overline{K})$ has the same order as $(\ker \psi)(\overline{K})$. 
 	Since $\psi$ is separable, its kernel is \'etale, so the order of $(\ker \psi)(\overline{K})$ is equal to the order of $\ker \psi$ as a group scheme, namely $\deg\psi$.
    
	As for the uniqueness part, assume that we have a decomposition $\phi = \psi \circ \eta$ as in the Proposition. We then have an exact sequence of group schemes over $K$
	$0 \rightarrow \ker \eta \rightarrow \ker \phi \rightarrow \ker \psi \rightarrow 0$, 
	where $\ker \eta $ is connected, and $\ker \psi$ \'etale (Proposition~\ref{prop:basics.isog}).
    The above is therefore a connected-\'etale sequence for $\ker\phi$ hence, by uniqueness of the latter, we have $\ker \eta = (\ker\phi)^\circ$ and $\ker \psi \simeq (\ker\phi)^{\textrm{\'et}}$.
    The uniqueness of $\eta$ and $\psi$, up to composition by automorphisms follows.
	\end{proof}

	\begin{prop}\label{prop:prime.power.decomposition}
	Any isogeny between abelian varieties factors as a composition of isogenies of prime power degree.
	\end{prop}
	\begin{proof}
	Let $\phi: A\to B$ be an isogeny of degree $d$ between abelian varieties. 
	Its kernel $G$ is a finite flat commutative group scheme of order $d$ over $K$.
	We claim that $G$ decomposes as a direct sum of its $\ell^{\infty}$-parts over the primes $\ell$: 
	\[G = \bigoplus_{\ell \text{ prime}} G[\,\ell^{v_\ell(d)}\,],\]
	where $G[\,k\,]$ denotes the kernel of multiplication by $k\geq 1$ on $G$, and $v_\ell$ denotes the $\ell$-adic valuation. 
    Indeed, if $d = m n$ with $m$ and $n$ coprime, write a B\'ezout relation $m\,u + n\,v = 1$.
	Then, the multiplication maps $[\,nv\,]$ and $[\,mu\,]$ project $G$  to $G[\,m\,]$ and $G[\,n\,]$, respectively, because  $[\,d\,]\,G=\{e_G\}$. 
	These induce an isomorphism of group schemes $G \simeq G[\,m\,] \oplus G[\,n\,]$ (in particular $G[\,m\,]$ has order $m$ and $G[\,n\,]$ has order $n$). 
	Proceeding by induction on the number of prime factors of $d$, one obtains the claimed decomposition. 
	
	Let $\ell$ be a prime factor of $d=\deg\phi$, the quotient group scheme $G' :=G/G[\ell^{v_\ell(d)}]$ is finite flat of order $d' = d/\ell^{v_\ell(d)}$.
	The isogeny $\phi : A\to B$ then factors through the quotient isogeny $\pi_\ell : A \to A'=A/G[\ell^{v_\ell(d)}]$ which has order $\ell^{v_\ell(d)}$: we thus obtain an isogeny $\phi' : A'\to B$ of degree $d'$ with kernel $G'$. 
	By immediate induction on the number of prime factors of~$d$, we obtain that $\phi$ factors as a composition $\pi_{\ell_1}\circ\pi_{\ell_2}\circ\dots\circ\pi_{\ell_r}$ of isogenies of prime power degrees. 
	\end{proof}
		
\subsection{Biseparable isogenies}
    
	\begin{defi}
	An isogeny   between abelian varieties is called {\it biseparable} if it is separable with separable dual isogeny.
    Equivalently (Proposition~\ref{prop:basics.isog}), an isogeny is biseparable if its kernel is \'etale and has \'etale (Cartier) dual.
	\end{defi}

If $\phi$ is a biseparable isogeny, any isogeny appearing as a factor of $\phi$ in the decomposition provided by Proposition~\ref{prop:prime.power.decomposition} is also biseparable. 
If the base field has characteristic $0$, all isogenies between abelian varieties are biseparable. 
In positive characteristic, we have the following characterization:
	\begin{prop}
    If the characteristic of $K$ is a prime $p$, an isogeny between abelian varieties over $K$ is biseparable if and only if its  degree is coprime to $p$. 
	\end{prop}
	\begin{proof}
    Let $\phi : A \rightarrow B$ an isogeny of degree $d$ between abelian varieties over $K$. If $p \nmid d$, the extension $K(A)/K(B)$ induced by $\phi^*$ has degree $d$, hence is automatically separable. By Proposition~\ref{prop:basics.isog}, $\phi$ is therefore separable and
    so is its dual (which also has degree $d$), so $\phi$ is biseparable.

    Conversely, assume that $\phi$ is biseparable. By Proposition~\ref{prop:prime.power.decomposition}, $\phi$ factors through an isogeny $\phi_p$ of degree $p^{v_p(d)}$, which is therefore also biseparable. The kernel $G$ of $\phi_p$ is thus a commutative \'etale group scheme of order $p^{v_p(d)}$ with \'etale Cartier dual $G^\vee$. Assume for a contradiction that $v_p(d)$ is positive. After a finite extension (which preserves \'etaleness and commutes with   formation of duals), there exists a point $x \in G(K)$ of exact order $p$. 
    This point  generates a copy of $(\Z/p\Z)_K$ in $G$, and there is a monomorphism $(\Z/p\Z)_K \rightarrow G$. 
    Taking Cartier duals, we obtain an epimorphism from the \'etale group scheme $G^\vee$ to $(\Z/p\Z)^\vee_K\simeq (\mu_p)_K$, so that $(\mu_p)_K$ is \'etale too. This  is a contradiction! We conclude that $v_p(d)$ must be $0$ \ie{}, that  $p \nmid d$.
	\end{proof}

We can now prove a general version of assertion \itemref{item:hdiff.a} in Theorem~\ref{itheo:hdiff}, which is probably well-known to the experts.
    
	\begin{theo}\label{theo:diffheight.bisep.isogenies}
	Let $\phi : A\to B$ an isogeny between abelian varieties over $K$. The following hold:
	\begin{enumerate}[{\rm(}a{\rm)}]
		\item If $\phi$ is biseparable, then \[\hdiff(A/K)=\hdiff(B/K).\]
		\item More generally, if $\phi$ is separable, then $\hdiff(A/K)\geq \hdiff(B/K)$.
	\end{enumerate}
	\end{theo}
	\begin{proof} 
		Denote the N\'eron models of $A$ and $B$ by $\pi_A :\Acal\to C$ and $\pi_B:\Bcal\to C$, respectively.
		We write $e_A : C\to\Acal$ and $e_B:  C \to \Bcal$ for their respective  neutral sections.
		Recall that $\Omega_A = e_A^\ast\Omega^1_{\Acal/C}$ and $\Omega_B = e_B^\ast\Omega^1_{\Bcal/C}$.
		We  let $\omega_A := \bigwedge^g\Omega_A$ and $\omega_B := \bigwedge^g\Omega_B$, and consider the line bundle $\omega(\phi) := \omega_A \otimes \omega_B^{-1}$. 
        By definition of $\hdiff(A/K)$ and $\hdiff(B/K)$, the theorem is equivalent to proving that $\deg\omega(\phi)=0$.
        We actually prove that $\omega(\phi)$ is a trivial line bundle on $C$. 
		
		By the N\'eron property, the isogeny $\phi:A\to B$ extends uniquely into a morphism $\Phi:\Acal\to\Bcal$ of group schemes. 
		The morphism $\Phi$ induces a map $\Phi^\ast\Omega^1_{\Bcal/C}\longrightarrow\Omega^1_{\Acal/C}$, which we may pull-back via the neutral section $e_A$ to a map of $\O_C$-modules. 
		Because $\Phi\circ e_A = e_B$, we thus obtain a morphism of $\O_C$-modules $\Omega_B = e_B^\ast\Omega^1_{\Bcal/C}\simeq e_A^\ast\Phi^\ast \Omega^1_{\Bcal/C} \longrightarrow e_A^\ast\Omega^1_{\Acal/C} = \Omega_A$.	
		Taking $g$-th exterior products, we deduce a morphism $\omega_B \to \omega_A$ and, tensoring by $\omega_B^{-1}$, we finally obtain a morphism $\theta : \O_C \to \omega(\phi)$. 
		To conclude, it suffices to prove that $\theta$ is, fiber-by-fiber, an isomorphism.
		
		We now argue Zariski locally on $C$: let $s$ be a closed point of $C$, denote the respective fibers of $\Acal$ and $\Bcal$ at $s$ by $\Acal_s$ and $\Bcal_s$, and let $\Phi_s : \Acal_s\to\Bcal_s$ be the restriction of $\Phi$ to  these fibers.
        By Proposition~\ref{prop:semistableisogenies}{\it (b)} above,  $\Phi_s$ is finite of degree $d$ and is surjective on the identity components of $\Acal_s$ and $\Bcal_s$ \ie{}, is an isogeny of degree $d$. 
		Since $d$ is coprime to the residual characteristic at $s$, Proposition~\ref{prop:semistableisogenies}{\it (c)} implies that the isogeny $\Phi_s$ is an \'etale morphism, hence it induces an isomorphism $\Phi_s^\ast \Omega^1_{\Bcal_s/k_s} \stackrel{\sim}{\longrightarrow} \Omega^1_{\Acal_s/k_s}$ (see \cite[\S2.2, Corollary 10]{BLR}). 
		Tracing through the definitions, we deduce that the restriction of $\theta$ to the fiber above $s$ is an isomorphism $\O_s\stackrel{\sim}{\longrightarrow}\omega(\phi)_s$. 
		This proves the first assertion. 
		
		Let $\phi : A\to B$ be a separable isogeny. 
		Since $\phi$ is \'etale, the morphism $\Lie(\Acal/C) \rightarrow \Lie(\Bcal/C)$ of Lie algebras induced by $\phi$ (by functoriality) is an isomorphism at the generic fiber. 
		Lemma~\ref{lem:keyinequalityparallelogram} then implies that  $\deg \Lie(\Acal/C) \leq \deg \Lie(\Bcal/C)$, so that $\hdiff(A/K) \geq \hdiff(B/K)$. This prove item $(b)$.
		\end{proof}
  	
	\begin{rema}
	Our proof of $(a)$ is  heavily inspired by the Raynaud's arguments in \cite[p. 202--203]{Ray}. 
	Note, though, that Raynaud assumes the base field to be either a number field or a discrete valued field with unequal characteristic. 
	His proof in the latter case can nevertheless be extended to our context  (where the characteristic of the function field is either $0$ or prime to the degree of the isogeny).
	\end{rema}

 We take advantage of the last theorem to relate the differential height of an abelian variety to that of its dual.
	\begin{prop}\label{prop:hdiff.dual}
	Let $A$ be an abelian variety over a function field $K=k(C)$.  
	If $K$ has positive characteristic, assume  that $A$ is semi-stable  over $K$. We have  $\hdiff(A^\vee/K) =\hdiff(A/K)$.
	\end{prop}
 	\begin{proof}
 	By construction, the abelian varieties $A$ and $A^\vee$ are isogenous over $K$ (for instance, apply Corollary~5 on p. 122 in \cite{MumfordAbVar08} to an ample line bundle on $A$).
 	If there exists a biseparable isogeny between $A$ and $A^\vee$, then $\hdiff(A^\vee/K) =\hdiff(A/K)$ by Theorem~\ref{theo:diffheight.bisep.isogenies}(a). (In particular, this settles the characteristic $0$ case of the proposition). 
	To treat the general case, recall from Proposition~\ref{prop:semistableisogenies} that, if $A$ is semi-stable  over $K$,  then so is its dual $A^\vee$.  In that situation, Lemme 2.4 in \cite[Chap.\,9]{MB85} states that the line bundle $(\bigwedge^g \Omega_A)\otimes(\bigwedge^g \Omega_{A^\vee})^{-1}$ has finite order in $\Pic(C)$. In particular, the degrees of $\Omega_A$ and $\Omega_{A^\vee}$ are equal,  hence the result.
 	\end{proof}

	\begin{rema}
	This proposition, when combined with   applications of Theorem~\ref{theo:diffheight.bisep.isogenies}(b)  to an isogeny and its dual,   yields a second proof of Theorem~\ref{theo:diffheight.bisep.isogenies}(a) in the semi-stable case.
	
    The reader wishing to extend Proposition~\ref{prop:hdiff.dual} to the non semi-stable case will have to deal with the behavior of base change conductors, with help from the case of discrete valuation rings: a recent preprint by Overkamp and Suzuki (see \cite[Theorem~1.2]{OverkampSuzuki23} in particular) proves that the base change conductors of an abelian variety and of its dual are equal. 
	\end{rema}

	\begin{rema}
    According to the Formule Cl\'e~\ref{theo-clef}, the modular height is proportional to the stable height. 
    One can thus deduce from Theorem~\ref{theo:diffheight.bisep.isogenies} analogous relations between the modular heights of suitably polarized abelian varieties which are linked by a (bi)separable isogeny.
	\end{rema}

\section{Frobenius and Verschiebung isogenies}\label{sec:FV}
{\it For the duration of this section, we work over a field  $F$ of characteristic $p>0$, and we assume that the (integral locally noetherian of absolute dimension $\leq 1$) scheme $S$ is of characteristic~$p$.}

The zoology of isogenies between abelian varieties over  a field such as $F$ is richer than over a field of characteristic~$0$,  in particular because Frobenius and Verschiebung isogenies appear.
In this section, we introduce the tools we will need in the following two sections in order to describe the variation of height in positive characteristic.
	
\subsection{Definitions}\label{Frobenius}
Let $F$ be a field of characteristic $p>0$.
Any scheme $S$ over $F$ is equipped with an absolute Frobenius morphism $\Frob_S : S\to S$ which acts trivially on points of $S$ and acts as the $p$-th power map on the structure sheaf $\O_S$. 
Let $X/S$ be an $S$-scheme with structural morphism $s_{X/S}:X\to S$. Using the cartesian diagram
	
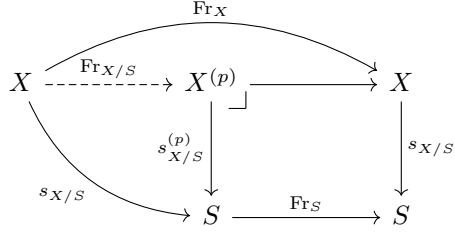
\begin{center}
	\begin{tikzcd}
			X \arrow[rrdd, "s_{X/S}"', bend right] \arrow[rrrr, "\Frob_{X}", bend left] 
			\arrow[rr, "\Frob_{X/S}", dashed] & & X^{(p)} \arrow[dd, "s^{(p)}_{X/S}"'] \arrow[rr, "\usebox\pullback"' , pos=-0.1, color=black] &  & X \arrow[dd, "s_{X/S}"] \\
			&&   &   &                         \\
			& & S \arrow[rr, "\Frob_S"]       &  & S 
	\end{tikzcd}
	\end{center}
one defines the so-called {\it relative Frobenius morphism} $\Frob_{X/S} : X\to X^{(p)}$. See \cite[Exposé VII$_A$, \S 4.1]{SGA3}.

\begin{exple}\label{exampleFrobenius}
		Let us write  what these different definitions of Frobenius explicitly mean for affine schemes. 
		Let $F$ be   field of a characteristic $p$ and $S = \Spec F$. Let  $X = \Spec R$ be an affine scheme with $R = F[X_1, \dots, X_m]/(f_1, \dots, f_n)$ for some $f_1, \dots, f_n \in F[X_1, \dots, X_m]$. 
		The absolute Frobenius on $R$ (resp. on $F$) is given by $x \mapsto x^p$ for all $x \in R$ (resp. $x\in F$). Thus $X^{(p)}=\Spec R^{(p)}$ where $R^{(p)}$ is the tensor product of the structural morphism $F \rightarrow R$ with the absolute Frobenius $\Frob_F: F \rightarrow F$. Explicitly, we have
		\[ R^{(p)} = F[X_1, \dots, X_m]/(f_1^{(p)}, \dots, f_n^{(p)}), \]
		where, for each $f \in F[X_1, \dots, X_n]$, $f^{(p)}$ is defined by raising each coefficient of $f$ to the power $p$. 
		The canonical map from $R \rightarrow R^{(p)}$ is induced by $f \mapsto f^{(p)}$ (this map is not $F$-linear). 
		The relative Frobenius $\Frob_{R/F} : R^{(p)} \rightarrow R$ is then given by $f \mapsto f(X_1^p, \cdots, X_m^p)$ (and is $F$-linear).
		
		In particular, assuming $f_1(0) = \dots = f_n(0)=0$, the cotangent map 
		$\mathfrak{M}_{R^{(p)}}/\mathfrak{M}_{R^{(p)}}^2 \rightarrow \mathfrak{M}_{R}/\mathfrak{M}_{R}^2$
		 induced by $\Frob_{X/F}$ at the maximal ideal $(X_1, \dots, X_m)$ is the zero map (and thus the tangent map is also trivial). 
		 Indeed, for each $f\in R^{(p)}$ with $f(0, \dots, 0)=0$, $f(X^p)$ vanishes at $(0,\dots, 0)$ with order at least $p\geq 2$. 
	\end{exple}
	
	\begin{defi}
	Let $A$ be an abelian variety over $F$. 
	The above construction defines a morphism ${\Frob_{A/F} : A \to A^{(p)}}$, which is an isogeny. 
	This morphism  is called the \emph{Frobenius isogeny} of $A$. 
	In what follows, the kernel of $\Frob_{A/F}$ will  be denoted by $A[\Frob_{A/F}]$.
	
	The dual abelian variety $A^\vee$ also possesses a Frobenius isogeny $\Frob_{A^\vee/F}: A^\vee \rightarrow (A^{\vee})^{(p)}$, the dual of which (see \cite[\S 7.3, Lemma 5]{BLR}) is called the \emph{Verschiebung isogeny} of $A$, and is denoted by $\Ver_{A/F}: A^{(p)}\to A$.  
	See also \cite[Exposé VII$_A$, \S4.3]{SGA3} for a definition in a more general context.		
	\end{defi}

	\begin{prop}
		\label{propFrobVer}
		Let $A$ be an abelian variety defined over $F$. The following hold:
			\vspace{-0.6\topsep}\begin{enumerate}[{\rm(}a{\,\rm)}]\setlength{\itemsep}{0em}
			\item The Frobenius isogeny $\Frob_{A/F} : A \to A^{(p)}$ is purely inseparable of degree $p^{\dim A}$.
			\item The Verschiebung isogeny $\Ver_{A/F} : A^{(p)} \to A$ has degree $p^{\dim A}$.
			\item We have $[\,p\,]_A = \Ver_{A/F} \circ \Frob_{A/F}$.
			\item We have $A^{(p)}\simeq A/A[\Frob_{A/F}]$.
		\end{enumerate}
	\end{prop}
	
	\begin{proof}
	Assertion $(d)$ is immediate since $\Frob_{A/F}$ is an isogeny $A \rightarrow A^{(p)}$ with kernel $A[\Frob_{A/F}]$.
	We prove $(a)$ next, making use of the computations in the above example.
	Choose an affine neighborhood $U = \Spec R$ of $0_A\in A$ and write 
	\[ R = F[X_1, \dots, X_m]/(f_1, \dots, f_n)\]
	with $f_1(0,\cdots, 0) = \cdots = f_n(0, \cdots, 0) = 0$ (so that $0_A$ has affine coordinates $(0,\dots, 0)$).
	The kernel of $\Frob_{A/F}$ is then $\Spec B$ with 
		\[ B = F[X_1, \dots, X_m]/(f_1, \dots, f_n, X_1^p, \dots, X_m^p).\]
	In particular, $B$ is finite-dimensional over $F$, hence $A[\Frob_{A/F}]$ is finite. 
	The only geometric point of this group scheme is $0_A$ itself, so $\Frob_{A/F}$ is purely inseparable (Proposition~\ref{prop:basics.isog}). 
	We write $\mathfrak{m}\subset R$ for the maximal ideal generated by the images $x_1, \dots, x_m$ of $X_1, \dots, X_m$ in $R$, and $\widehat{R}$ for the completion of $R$ with respect to the ideal $\mathfrak{m}$. 
	The ring $\widehat{R}$ is a complete regular local ring of dimension $g=\dim A$, so there is an isomorphism 
		$ \varphi: \widehat{R} \stackrel{\sim}{\longrightarrow} F[[T_1, \cdots, T_g]]$.
				
	Furthermore, we can choose  the local parameters of $R$ at $0_A$ such that its first $g$ elements $x_1, \cdots, x_g$ make up a basis of $\mathfrak{m}/\mathfrak{m}^2$. 
	The map $\varphi$ then  sends $x_i$ to $T_i$ for $i \in \{1, \cdots, g\}$. 
	Now, as $\mathfrak{m}$ contains all powers of the $x_i$'s:
	\begin{align*}
	B = R \otimes F[X_1, \cdots, X_m]/(X_1^p, \cdots, X_m^p)  
	&\simeq \widehat{R} \otimes  F[X_1, \cdots, X_m]/(X_1^p, \cdots, X_m^p)  \\
	&\simeq  F[[X_1, \cdots, X_g]]/(X_1^p, \cdots, X_g^p) 
	= F[X_1, \cdots, X_g]/(X_1^p, \cdots, X_g^p).
	\end{align*}
	It is apparent that $B$ is an $F$-algebra of  dimension $p^g$, and therefore that $A[\Frob_{A/F}]$ is a groups scheme of order $p^g$.
		
	Assertion $(b)$ now follows, by definition of $\Ver_{A/F}$ as the dual of the Frobenius isogeny of $A^\vee$.
	Finally, item $(c)$ is a direct consequence of the construction of the Verschiebung isogeny explained in \cite[Exposé VII$_A$, \S4.3]{SGA3} (which coincides with ours in this context). 
	\end{proof}

	\begin{lemm}\label{lemm:faithfulflatnessFrobeniussmoothgroups}
    Let $\Gcal$ be a smooth group scheme of finite type over $S$ of characteristic $p$.
    The Frobenius morphism $\Frob_{\Gcal/S} : \Gcal \rightarrow \Gcal^{(p)}$ is faithfully flat.
	\end{lemm}
	\begin{proof}
	We argue, as we may, Zariski locally on $S$. 
	Assume that $S = \Spec F$ for some field $F$, and that $G$ is a smooth algebraic group over $F$. 
	Then $G^{(p)}$ is also smooth (in particular it is reduced), and the Frobenius map $\Frob_{G/F}: G \rightarrow G^{(p)}$ is surjective on closed points 
	(because the Frobenius field morphism $x\mapsto x^p$ has an inverse on $\overline{F}$). 
	
	Since $G^{(p)}$ is reduced,  we conclude by \cite[Proposition 1.70]{Milne} that $\Frob_{G/F}$ is faithfully flat.
	\end{proof}

The following notion will be useful in the paper.

\begin{defi}
Let $\phi$ be an isogeny between abelian varieties defined over a function field $K$ of characteristic $p>0$.
We call $\phi$ a {\it FV-isogeny} if $\phi$ factors as a composition (in any order) of a certain number $e\geq 0$ of Frobenius isogenies, a certain number $f\geq 0$ of Verschiebung isogenies, and any number of biseparable isogenies. 
In which case, we say that this FV-isogeny is {\it of type $(e,f)$}.   
\end{defi}

It follows from the definitions that the class of FV-isogenies behaves well with respect to duality: if $\phi$ is a FV-isogeny of type $(e,f)$, its dual $\widehat{\phi}$ is also a FV-isogeny and the type of $\widehat{\phi}$ is $(f,e)$. 
Note that in dimension $1$, any isogeny between non-isotrivial elliptic curve is a FV-isogeny, see Proposition 4.7 in \cite{GriPaz}.

\subsection[Ordinarity and p-rank]{Ordinarity and $p$-rank}	
	
	\begin{defi}\label{defi:ordinary} 
	Let $A$ be an abelian variety over a field $F$ of positive characteristic $p$. 
	We say that $A$ is {\it ordinary} if $A[\,p\,](\overline{F}) \simeq (\Z/p\Z)^{\dim A}$ (\ie{}, $A[\,p\,](\overline{F})$ is as large as it possibly can). 
	This condition is equivalent to the separability of $\Ver_{A/F} : A^{(p)}\to A$ since the multiplication-by-$p$ maps on $A$ factors as $[\,p\,] = \Ver_{A/F} \circ \Frob_{A/F}$ by Proposition~\ref{propFrobVer}, and because $\Frob_{A/F} :A \to A^{(p)}$ is universally injective.
			
    The \emph{$p$-rank} of $A$ is the integer $f$ such that $A[\,p\,](\overline{F}) \simeq (\Z/p\Z)^f$.
    This integer lies between $0$ and $\dim A$. 
    Note that $A$ is ordinary exactly when its $p$-rank is $\dim A$. 
    Notice also that $p^f$ is the separability degree of $\Ver_{A/F}$. 
	\end{defi}

It follows directly from Definition~\ref{defi:ordinary} that, if $A/F$ is an ordinary abelian variety, then $(A\times_FF')/F'$ is ordinary for any field extension $F'/F$. 
Also, any abelian subvariety of an ordinary abelian variety is ordinary.

It is well-known (see \cite[p. 137]{MumfordAbVar08}) that the $p$-rank is constant in an isogeny class.
In particular, if two abelian varieties over $F$ are isogenous, then they are either both ordinary or neither of them is.

\subsection[{Frobenius height of p-group schemes and isogenies}]{Frobenius height of $p$-group schemes and isogenies}
\label{sec:Frob.height}
    
	\begin{defi}\label{defi:Frobheight}
	Let $\Gcal$ be a finite flat group scheme over $S$ of characteristic $p$.  
	For any integer $d\geq 1$, define the iterated Frobenius $\Frob_{\Gcal/S}^{(d)} : \Gcal \longrightarrow \Gcal^{(p^d)}$ as the composition 
	\[\Frob_{\Gcal/S}^{(d)}: \Gcal \xrightarrow[]{\ \Frob_{\Gcal/S}\ } \Gcal^{(p)} 
	\xrightarrow[]{\ \Frob_{\Gcal^{(p)}/S}\ }  \Gcal^{(p^2)} 
	\xrightarrow[]{\  \ } \cdots 
	\xrightarrow[]{\  \ } \Gcal^{(p^{d-1})} 
	\xrightarrow[]{\  \Frob_{\Gcal^{(p^{d-1})}/S}\  } \Gcal^{(p^d)}.\]
	We also let $\Frob_{\Gcal/S}^{(0)} : \Gcal \to\Gcal$ be the trivial morphism. 
	Consider the increasing sequence of closed subgroup schemes
	\[ \Gcal[\Frob_{\Gcal/S}] \subset \Gcal[\Frob_{\Gcal/S}^{(2)}] \subset \cdots \subset \Gcal[\Frob_{\Gcal/S}^{(n)}] \subset \Gcal[\Frob_{\Gcal/S}^{(n+1)}]\subset \cdots \subset \Gcal. \]

	We say that $\Gcal$ has {\it finite Frobenius height} if there exists an integer $d\geq 0$ such that $\Gcal[\Frob_{\Gcal/S}^{(d)}] = \Gcal$. 
	When $\Gcal$ has finite Frobenius height, we denote by $\delta_p(\Gcal)$ the smallest such index $d\geq 0$.
	In other words, $\Gcal$ has Frobenius height $d$ if and only if $d$ is the smallest integer such that $\Frob_{\Gcal/S}^{(d)}:\Gcal\to\Gcal^{(p^d)}$ is trivial on $\Gcal$.
	\end{defi}
 
	\begin{rema}
	Other authors use the term `height' instead of our `Frobenius height', 
	we preferred our less standard terminology in order to avoid obvious confusion. 
	Moreover, if $\Gcal$ and $\Hcal$ are two subgroup schemes of a fixed group scheme over $S$, 
	and if both have finite Frobenius height, then one easily checks that
		\begin{equation}\label{eq:frobh.minmax}
		    \delta_p(\Gcal+\Hcal)=\max\{\delta_p(\Gcal), \delta_p(\Hcal)\}, \qquad
		\delta_p(\Gcal\cap\Hcal)\leq \min\{\delta_p(\Gcal), \delta_p(\Hcal)\}
		\end{equation}
  	Note that the only group scheme of Frobenius height $0$ is the trivial group scheme.
	\end{rema}

	\begin{prop}\label{prop:finite.frob.height}
	A finite group scheme over $F$ is connected if and only if it has finite Frobenius height.
	\end{prop}
	\begin{proof}
	Let $G$ be a finite group scheme over $F$ (so that, in particular, $G$ is affine). 
    Applying the connected-\'etale sequence to $G$ (Lemma~\ref{lemm:connectedetalesequence}), 
    and using the functoriality of the Frobenius, we get a  commutative diagram  
	\[\xymatrix{
			0 \ar[r] & G^{\circ} \ar[r] \ar[d]_{\Frob_{G^{\circ}/F}} & G \ar[r] \ar[d]_{\Frob_{G/F}} & G^{\textrm{\'et}} \ar[d]^{\Frob_{G^{\textrm{\'et}}/F}} \ar[r] & 0 \\
			0 \ar[r] & (G^\circ)^{(p)} \ar[r] & G^{(p)} \ar[r] & (G^{\textrm{\'et}})^{(p)} \ar[r] & 0\,,}
	\]	
	where $G^{\circ}$ is a connected subgroup scheme of $G$, and  $G^{\textrm{\'et}}$  is an \'etale group scheme over $F$. 
	The Frobenius map is surjective on the `\'etale part' $G^{\text{\'et}}$  by \cite[Corollaire 8.3.1]{SGA3}.
	Iterating the argument, if $G^{\textrm{\'et}} \neq \{0\}$, no iterate of the Frobenius is trivial on the whole group scheme $G$. This proves that a group scheme $G$ of finite Frobenius height must have   trivial \'etale part $G^{\textrm{\'et}}$  and thus, must be connected. 

  	To prove the converse, recall that, if $G$ is connected and nontrivial, $\Frob_{G/K}$ has nontrivial kernel (see \cite[Proposition, p. 50]{CornellSilverman} for instance), then use induction on the order of $G$.
	\end{proof}	

The following lemma will be used in section~\ref{sec:proof of theorem} to perform induction on the  Frobenius height: 
	\begin{lemm}\label{lemm:frob.height.lowering}
	Let $d\geq 2$ be an integer. Let $G$ be a finite group scheme of Frobenius height $d$ over $F$.
	Then the quotient group scheme  $G/G[\Frob_{G/F}]$ has Frobenius height $d-1$.
	\end{lemm}

	\begin{proof}
	Let $G_1:=G[\Frob_{G/F}]$ be the kernel of $\Frob_{G/F} : G \to G^{(p)}$. 
	Then $G_1$ has Frobenius height $\leq 1$ (because the Frobenius morphism of $G_1$ is the restriction of $\Frob_{G/F}$), 
	and this group scheme fits into the exact sequence of group schemes over~$F$:
	\[0 \longrightarrow G_1 \longrightarrow G \longrightarrow G/G_1 \longrightarrow 0.\]
	The quotient $G/G_1$ is thus isomorphic to the image $\Frob_{G/F}(G)$, hence has Frobenius height $d-1$. 
	Indeed, by hypothesis, the iterated Frobenius $\Frob_{G/F}^{(d)}$ is trivial on $G$, so $\Frob_{G/F}^{(d-1)}$ is trivial on the image of $\Frob_{G/F}$. 
	\end{proof}

Proposition~\ref{prop:finite.frob.height} above (combined with Proposition~\ref{prop:basics.isog}) allows to introduce the following notion:
	\begin{defi}
	Let $\phi$ be an isogeny between abelian varieties over a field $F$ of characteristic $p>0$. 
	Let $\phi\ins$ denote the purely inseparable part of~$\phi$, as defined in Proposition~\ref{prop:isogeny.factor.sep.insep}. 
	We define the {\it Frobenius height $\delta_p(\phi)$ of $\phi$} to be the Frobenius height of the kernel of $\phi\ins$: \[\delta_p(\phi) := \delta_p(\ker\phi\ins).\]
	\end{defi}

	\begin{rema} 
    \label{rema:subadditivitydeltap}
	By definition, the Frobenius isogeny $\Frob_{A/F} : A\to A^{(p)}$ of an abelian variety has Frobenius height $1$.
	For any isogeny $\phi$ between abelian varieties of dimension $g$, the following divisibility relations  hold:
	\[p^{\delta_p(\phi)}\mid \deg(\phi\ins)\mid p^{g\, \delta_p(\phi)},\]
    and we have $0 \leq \delta_p(\phi) \leq \log_p (\deg\phi\ins) \leq \max\{\log_p(\deg\phi), g\, \delta_p(\phi)\}$.
    
     It follows from the definition that the Frobenius height is `subadditive' in the following sense: for any two isogenies $\phi:A \rightarrow B$ and $\psi : B \rightarrow C$, we have
      $\delta_p(\psi \circ \phi) \leq \delta_p(\psi) + \delta_p(\phi)$. 

    Note also that an isogeny  $\phi$ is biseparable if and only if $\delta_p(\phi) + \delta_p(\widehat{\phi}) = 0$.
	\end{rema}

The Frobenius height of a FV-isogeny between ordinary abelian varieties admits a simple interpretation:

	\begin{lemm}\label{lemm:frob.height.FV}
	Let $K$ be a function field of positive characteristic $p$.
	Let $\phi:A\to B$ be a FV-isogeny of type $(e,f)$ between ordinary abelian varieties over $K$.
        Then  $e = \delta_p(\phi)$ and $f =\delta_p(\widehat{\phi})$.
	\end{lemm}
	\begin{proof} 
        Observe that a Frobenius isogeny   has Frobenius height $1$ (by definition), and that a separable isogeny has Frobenius height $0$  (since its kernel is étale).
        By the ordinarity assumption, all Verschiebung isogenies appearing in the decomposition of the FV-isogeny $\phi$ are separable. The subadditivity  of $\delta_p$ (Remark \ref{rema:subadditivitydeltap}) directly implies  that $\delta_p(\phi)\leq e$. 
       To prove the converse inequality we remark that, by definition, 
       $\phi$ factors as a composition
        \[A=A_0 \xrightarrow[]{\phi_1} A_1 \xrightarrow[]{\phi_2} A_2 \xrightarrow[]{\phi_3} \dots  \xrightarrow[]{\phi_{r-1}}A_{r-1}\xrightarrow[]{\phi_r} A_r=B.\]
        The multiplicativity of the inseparability degrees and  the fact that the Verschiebung isogenies among the $\phi_i$'s are separable implies that 
        \[
        \deg\ins(\phi) = \prod_{i=1}^r \deg\ins(\phi_i) =(\deg\Frob_{A/K})^e =  p^{e \dim (A)}.
        \]
        On the other hand, $\deg(\phi\ins)$ is  the order of the connected component $G^\circ$ of $G=\ker\phi$ (see Proposition~\ref{prop:isogeny.factor.sep.insep}).
        By construction of the Frobenius height, $G^\circ$ is contained in $A[\Frob_{A/K}^{(\delta_p(\phi))}]$. 
        Hence its order $|G^\circ|=\deg(\phi\ins)$ divides  $p^{\delta_p(\phi)\,\dim A} = \big|A[\Frob_{A/K}^{(\delta_p(\phi))}]\big|$. This proves that $e\leq\delta_p(\phi)$ and, thus, the first equality of the lemma. 

        The second equality follows upon noting that the dual $\widehat{\phi}:\ \widehat{B} \rightarrow \widehat{A}$ of $\phi$ is a FV-isogeny of type $(f,e)$.         
	\end{proof}

The above lemma yields a direct relation between the results of the present paper and those in \cite{GriPaz}:
	\begin{lemm}\label{lemm:frob.height.dim1}
	Let $K$ be a function field of positive characteristic $p$.
	Let $\phi:E_1\to E_2$ be an isogeny between two non-isotrivial elliptic curves over $K$. 
    Then $p^{\delta_p(\phi)} = \degins(\phi)$ and $p^{\delta_p(\widehat{\phi})} = \degins(\widehat{\phi})$.
	\end{lemm}
	\begin{proof} 
	Proposition 4.7 in \cite{GriPaz} shows that the isogeny $\phi$ factors as
	$\phi = \Ver_{E_2/K}^{(f)}\circ \psi\circ\Frob_{E_1/K}^{(e)}$, where $\psi$ is biseparable, and the integers $e$ and $f$ satisfy $p^e = \deg\ins(\phi)$ and $p^f=\deg\ins(\widehat{\phi})$.
	In other words, $\phi$ is a FV-isogeny of type $(e,f)$ with $e=\log_p\deg\ins(\phi)$ and $f=\log_p\deg\ins(\widehat{\phi})$.
      The non-isotriviality of $E_1$ and $E_2$ implies that both of these elliptic curves are ordinary. 
    Lemma \ref{lemm:frob.height.FV} allows to conclude. 
	\end{proof}

\subsection[Hodge bundles and group schemes of Frobenius height 1]{Hodge bundles and group schemes of Frobenius height $1$}
The goal of this paragraph is to relate the Hodge bundles of a group scheme $\Gcal$ and of its `Frobenius twist' $\Gcal^{(p)}$ in  case  $\Gcal$ has Frobenius height $1$. 
	
	\begin{lemm}\label{lemexactsequenceLie}
	The functor $\Gcal \mapsto \Lie(\Gcal/S)$ is left-exact on group schemes over $S$. 
	In other words, if we have a left exact sequence  $0\longrightarrow\Gcal_1\longrightarrow\Gcal_2\longrightarrow\Gcal_3 $ of group schemes over $S$, then the induced sequence of $\O_S$-modules
	\[0\longrightarrow\Lie(\Gcal_1/S)\longrightarrow\Lie(\Gcal_2/S)\longrightarrow\Lie(\Gcal_3/S)\]
		is exact.
    \end{lemm}
	\begin{proof}
	By functoriality, the maps involved in $0\rightarrow\Lie(\Gcal_1/S)\rightarrow\Lie(\Gcal_2/S)\rightarrow\Lie(\Gcal_3/S)$ are well-defined. 
	To prove that this sequence of $\Ocal_S$-modules is exact, it suffices to show that it is fiberwise exact.
	For any point $s\in S$ with residue field $k_s$, consider the following commutative diagram 
	\begin{center}
	\begin{tikzcd}
				& 0 \arrow[d, dotted] & 0 \arrow[d] & 0 \arrow[d] &   \\
				0 \arrow[r] & \Lie(\Gcal_1/S)_{k_s} \arrow[r] \arrow[d, dotted] & {\Gcal_1(k_s[\varepsilon])} \arrow[r] \arrow[d] & \Gcal_1(k_s) \arrow[r] \arrow[d] & 0 \\
				0 \arrow[r] & \Lie(\Gcal_2/S)_{k_s} \arrow[r] \arrow[d, dotted] & {\Gcal_2(k_s[\varepsilon])} \arrow[r] \arrow[d] & \Gcal_2(k_s) \arrow[r] \arrow[d] & 0 \\
				0 \arrow[r] & \Lie(\Gcal_3/S)_{k_s} \arrow[r]                   & {\Gcal_3(k_s[\varepsilon])} \arrow[r]           & \Gcal_3(k_s) \arrow[r]           & 0\,.
	\end{tikzcd}
	\end{center}
	Here $k_s[\varepsilon]$ denotes the $k_s$-algebra of dual numbers over $k_s$. 
	The horizontal sequences are exact, by definition of Lie groups, and the two right-most vertical sequences are exact by definition of exact sequence of group schemes.
	Simple diagram-chasing shows that the left-most vertical arrows form an exact sequence of $k_s$-vector spaces.
		
	(An almost identical result is stated in \cite[II, \S 4, 1.5]{DemazureGabriel}, in terms of Lie group functors.)
	\end{proof}
	
    \begin{prop}\label{prop:lie.ker.frob}
	For any group scheme $\Gcal$ over $S$ such that $\Omega(\Gcal/S)$ is locally free of finite type, 
		there are canonical isomorphisms of $\Ocal_S$-modules 
	\[ \Omega(\Gcal/S) \simeq  \Omega(\Gcal[\Frob_{\Gcal/S}]/S), \quad \text{ and }\quad \Lie(\Gcal/S)  \simeq \Lie(\Gcal[\Frob_{\Gcal/S}]/S).\]	 	
    \end{prop}
	\begin{proof}
	We start with the left exact sequence of group schemes of finite type over $S$: 
	\[ 0 \longrightarrow  \Gcal[\Frob_{\Gcal/S}] \longrightarrow \Gcal \xrightarrow[]{\Frob_{\Gcal/S}}     \Gcal^{(p)}.\]
	By \cite[Proposition 6.1.24]{Liu02},  the contravariant Hodge bundle functor from pointed $S$-schemes to $\Ocal_S$-modules is right-exact. Hence an exact sequence of $\O_S$-modules:
	\[\xymatrix{\Omega(\Gcal^{(p)}/S) \ar[r]& \Omega(\Gcal/S) \ar[r] & \Omega(\Gcal[\Frob_{\Gcal/S}]/S) \ar[r] & 0}.\]
	 The left-most arrow in this sequence is the cotangent map of the Frobenius $\Frob_{\Gcal/S}$ along the zero section of $\Gcal$, 
	it is thus trivial (see Example~\ref{exampleFrobenius}). This yields the desired isomorphism between Hodge bundles.
	
	The isomorphism between Lie groups is obtained by the exact same argument, using the previous result. 
	\end{proof}

When one restricts to group schemes of Frobenius height $\leq 1$, Lemma~\ref{lemexactsequenceLie} can be strengthened:

	\begin{prop} \label{prop:exactness.lie.functor}
	The functor $\Gcal \mapsto \Lie(\Gcal/S)$ is exact on finite flat group schemes over $S$ of Frobenius height $\leq 1$. In other words,  assume that we have an exact sequence  (see Definition~\ref{defi:exact.sequence.groupschemes}) $0 \xrightarrow[]{\ \ } \Gcal_1 \xrightarrow[]{\ f\ } \Gcal_2 \xrightarrow[]{\ g\ } \Gcal_3 \xrightarrow[]{\ \ } 0$, where $\Gcal_1, \Gcal_2,\Gcal_3$ are finite flat group schemes over $S$ of Frobenius height $\leq 1$. 
    Then, the induced complex of $\Ocal_S$-modules 
	\begin{equation}\label{eq:complex.Lie}
    0 \xrightarrow[]{\ \ } \Lie(\Gcal_1/S) \xrightarrow[]{\ \Lie(f)\ } \Lie(\Gcal_2/S) \xrightarrow[]{\ \Lie(g)\ } \Lie(\Gcal_3/S) \xrightarrow[]{\ \ } 0
	\end{equation}    
	is   exact.
	\end{prop}
	\begin{proof} 
	Lemma~\ref{lemexactsequenceLie} already shows that $\Gcal \mapsto \Lie(\Gcal/S)$ is left-exact, so we only need to prove that $\Lie(g)$ is surjective.

	The three $\Ocal_S$-modules $\Lie(\Gcal_i/S)$ are locally free of finite rank, so it is enough to prove that the complex is exact at every closed point $s \in S$. Let us fix such a point $s \in S$, with residue field $k_s$.
	The  sequence of finite group schemes of height $\leq 1$ over $k_s$
	\[0 \xrightarrow[]{\ \ }  (\Gcal_1)_s \xrightarrow[]{\ f\ } (\Gcal_2)_s \xrightarrow[]{\ g\ }(\Gcal_3)_s \xrightarrow[]{\ \ } 0,\]
	is exact. Applying the $\Lie$ functor gives rise to the complex of $k_s$-algebras
	\begin{equation}\label{eq:complexLieAlg}
		0 \xrightarrow[]{\ \ } \Lie((\Gcal_1)_s/k_s) \xrightarrow[]{\ \Lie(f)\ } \Lie((\Gcal_2)_s/k_s) \xrightarrow[]{\ \Lie(g)\ } \Lie((\Gcal_3)_s/k_s) \xrightarrow[]{\ \ } 0.
	\end{equation}
    The latter is exactly the pullback over $s$ of the complex of $\Ocal_S$-modules in \eqref{eq:complex.Lie}.  
    An immediate consequence of the structure theorem for group schemes of height $\leq 1$ (see \cite[Th\'eor\`eme 4.2, p. 334]{DemazureGabriel}) is the fact that,
    for any finite flat group scheme $G$ of Frobenius height $\leq 1$ over a field $k$, the order of $G$ is equal to $p^n$ where $n = \dim \Lie(G/k)$.

    Write  $p^{n_1}$, $p^{n_2}$, and $p^{n_3}$ for the orders of $\Gcal_1$, $\Gcal_2$, and $\Gcal_3$, respectively (recall that these group schemes are flat over $S$). 
    The dimensions of their Lie algebras over $k_s$ are thus $n_1$, $n_2$, and $n_3$, respectively.  
    By exactness of the group scheme sequence, we have $n_2 = n_1 + n_3$. 
    Therefore the complex of Lie algebras in \eqref{eq:complexLieAlg}, which is already exact on the left,  must also be exact on the right by comparison of dimensions.
    This concludes the proof.
	\end{proof}
	
	\begin{prop}\label{prop:Hodgebundleofptwist}
	Let $\Gcal$ be a group scheme  over $S$ with structural morphism $\pi: \Gcal \to S$ and neutral section $e:S\to\Gcal$.
	There is a canonical isomorphism of Hodge bundles over $S$: 
	\[\Omega(\Gcal^{(p)}/S)\simeq \Frob_{S}^\ast\, \Omega(\Gcal/S).\]    
	\end{prop}
	\begin{proof}
	By construction of $\Frob_{\Gcal/S}$, the $S$-group scheme $\Gcal^{(p)}$ is the fiber product of $\pi : \Gcal\to S$ with $\Frob_S:S\to S$.
	Hence its structural morphism $\pi^{(p)}: \Gcal^{(p)}\to S$ has neutral section $e^{(p)} = \Frob_{\Gcal/S}\circ e$. 
	Here is a diagram of the situation:
	\begin{center}\begin{tikzcd}
				\mathcal{G} \arrow[rddd, "\pi"] \arrow[rrd, "\mathrm{Fr}_{\mathcal{G}/S}"] \arrow[rrrrd, "\mathrm{Fr}_{\mathcal{G}}", bend left=20] &                &  &   &  \\
				& & \mathcal{G}^{(p)} \arrow[rr, "f_p"] \arrow[ldd, "\pi^{(p)}", pos=0.2] &  & \mathcal{G} \arrow[ldd, "\pi"'] \\
				& &  & &                                 \\
				& S \arrow[rr, "\mathrm{Fr}_{S}"] \arrow[luuu, "e", dashed, bend left] \arrow[ruu, "e^{(p)}", dashed, bend left, pos=0.8] &  & S \arrow[ruu, "e"', dashed, bend right] &    
	\end{tikzcd}\end{center}
	Notice that $f_p\circ e^{(p)} = e\circ \Frob_S$, and \cite[Proposition 6.1.24]{Liu02} yields a canonical isomorphism $\Omega^1(\Gcal^{(p)}/S) \simeq f_p^\ast\,\Omega^1(\Gcal/S)$.
	By definition of the Hodge bundles of $\Gcal^{(p)}/S$ and $\Gcal/S$, we now have
	\[\Omega(\Gcal^{(p)}/S) 
			= (e^{(p)})^\ast\,\Omega^1(\Gcal^{(p)}/S) 
			\simeq (f_p\circ e^{(p)})^\ast\,\Omega^1(\Gcal/S) 
			= (e \circ \Frob_S)^\ast\,\Omega^1(\Gcal/S) 
			= \Frob_S^\ast\big(e^\ast\,\Omega^1(\Gcal/S)\big) 
			= \Frob_S^\ast\,\Omega(\Gcal/S).\]
	\end{proof}
	
\section{Hodge bundle and heights}\label{sec:Hodge and heights}
\emph{We now restrict our attention to the case where the base function field $K$ has positive characteristic $p$, and $S$ is a scheme over $\F_p$. Moreover, all  the isogenies we consider in this section have degree a power of $p$.}

Our main result in this section is Corollary~\ref{VerFrob} which describes the variation of differential height through a FV-isogeny.
	
\subsection{Lie algebras of quotient group schemes}
We state and prove a variant of Theorem 2.9 of \cite{Yua}:
	\begin{prop}\label{prop:thmprincipalYuan}
    Let $\Acal$ be a smooth flat commutative group scheme of finite type over $S$, 
    $\Gcal$ be a finite flat subgroup scheme of $\Acal$ of Frobenius height $1$,
    and  $\Bcal' := \Acal/\Gcal$ denote the \FPPF{} quotient group scheme over $S$. 
    We then have an exact sequence
    \[0 \longrightarrow \Lie(\Gcal/S) \longrightarrow\Lie(\Acal/S) \longrightarrow\Lie(\Bcal'/S)\longrightarrow\Frob_S^*\Lie(\Gcal/S) \longrightarrow 0\, .\]
	\end{prop}
	\begin{proof}
    We essentially rewrite the proof of \cite{Yua} here, in particular because its assumptions are slightly different from ours, but the majority of his arguments transpose smoothly to our setting. 

    First, by construction of $\Bcal'$ and by definition of Frobenius morphisms, the  commutative diagram in the (abelian) category of \FPPF{} group schemes over $S$
    \[\xymatrix{
	0 \ar[r] & \Gcal \ar[r] \ar[d]_{\Frob_{\Gcal/S}} & \Acal \ar[d]^{\Frob_{\Acal/S}} \ar[r] &\Bcal' \ar[d]^{\Frob_{\Bcal'/S}} \ar[r] & 0 \\
	0 \ar[r] & \Gcal^{(p)} \ar[r] & \Acal^{(p)} \ar[r] &(\Bcal')^{(p)} \ar[r] & 0,}\]
    has exact rows.
    Applying the snake lemma yields a long exact sequence 
    \[ 0 \longrightarrow\Gcal[\Frob_{\Gcal/S}] \longrightarrow \Acal[\Frob_{\Acal/S}] \longrightarrow\Bcal'[\Frob_{\Bcal'/S}] \longrightarrow \operatorname{coker}(\Frob_{\Gcal/S})\longrightarrow \operatorname{coker}(\Frob_{\Acal/S}) \longrightarrow \operatorname{coker}(\Frob_{\Bcal'/S}) \longrightarrow \cdots \,.\]
    Since $\Gcal$ has Frobenius height 1, we have $\Gcal[\Frob_{\Gcal/S}] = \Gcal$, and the cokernel of $\Frob_{\Gcal/S}$ is $\Gcal^{(p)}$ itself, while the cokernel of $\Frob_{\Acal/S}$ is  trivial because $\Acal$ is smooth (see Lemma~\ref{lemm:faithfulflatnessFrobeniussmoothgroups}).
    The above long exact sequence thus simplifies to the exact sequence:
    \[ 0 \longrightarrow\Gcal \longrightarrow \Acal[\Frob_{\Acal/S}] \longrightarrow \Bcal'[\Frob_{\Bcal'/S}]\longrightarrow \Gcal^{(p)} \longrightarrow 0\, .\]
    The four terms in this sequence are all finite flat group schemes of Frobenius height $\leq 1$ : exactness of the Lie functor on those group schemes (Proposition~\ref{prop:exactness.lie.functor}) implies that the  sequence 
    \[ 0 \longrightarrow \Lie(\Gcal/S)  \longrightarrow \Lie(\Acal[\Frob_{\Acal/S}]/S)  \longrightarrow\Lie(\Bcal'[\Frob_{\Bcal'/S}]/S)  \longrightarrow \Lie(\Gcal^{(p)}/S) \longrightarrow 0 \]
    is exact. 
    Finally, $\Lie(\Acal[\Frob_{\Acal/S}]/S)$ and $\Lie(\Bcal'[\Frob_{\Bcal'/S}]/S)$ canonically identify to $\Lie(\Acal/S)$ and $\Lie(\Bcal'/S)$, respectively  (see Proposition~\ref{prop:lie.ker.frob}).
	Recalling from Proposition~\ref{prop:Hodgebundleofptwist} that $\Lie(\Gcal^{(p)}/S)\simeq\Frob_S^*\Lie(\Gcal/S)$  concludes the proof.
	\end{proof}

	\begin{rema}
	At first sight, one could hope to partly extend Proposition~\ref{prop:thmprincipalYuan} to arbitrary finite flat subgroup schemes $\Gcal$ of $\Acal$ by writing $\Gcal[\Frob_{\Gcal/S}]$ instead of $\Gcal$ in the exact sequence, and possibly adding some errors terms (or not requiring exactness on the right). 
	However, as the following example shows, $\Gcal[\Frob_{\Gcal/S}]$ is not necessarily flat (even if $\Gcal$ is). 

	Let $K=\F_3(t)$, and consider the Igusa curve $E : y^2 = x^3+x^2-t^{-3}$ over $K$.
	One checks that this elliptic curve has $j$-invariant $t^3$, and that the point $P=(-t^{-1}, t^{-1})$ has exact order $3$ on $E$.
	Let $\pi:\Ecal\to\P^1$ denote the N\'eron model of~$E$.
	The fiber above $t=0$ of $\pi$ is a  supersingular elliptic curve over $\F_3$ (with $j$-invariant $0$). 
	Hence the image of $P$ in this fiber must be the neutral element.
	Let $\Gcal$ denote the Zariski closure in $\Ecal$ of the subgroup of $E$ generated by $P$ ($\Gcal$ is then a finite flat subgroup of $\Ecal$). 
	The generic fiber of $\Gcal$ is isomorphic to $(\Z/3\Z)_K$ (which is \'etale), and its fiber at $t=0$ is isomorphic to $\mmu_3$ (which is connected).
	The previous computation shows that the subgroup scheme $\Gcal[\Frob_{\Gcal/\P^1}]$ of $\Gcal$ is trivial at the generic fiber, but has order $3$ at $t=0$. 
	In particular, $\Gcal[\Frob_{\Gcal/S}]$  is not flat over $\P^1$!
	\end{rema}
 
Proposition~\ref{prop:thmprincipalYuan} applies to our height computation as follows.
	\begin{prop}\label{prop:Yuan}
    Assume that $S = C$ is a smooth projective curve over a perfect field $k$ of characteristic $p>0$.
    Let $\Acal$ be a smooth commutative group scheme of finite type over $C$, $\Gcal$ be a finite flat subgroup scheme of $\Acal[\Frob_{\Acal/S}]$ of Frobenius height $\leq 1$, and $\Bcal' := \Acal/\Gcal$ be the (\FPPF{}) quotient group scheme. Then, we have 
	\begin{equation*}\label{eq:yuan}
    \deg \Omega({\Bcal'/C}) = \deg \Omega({\Acal/C}) - (p-1) \deg \Lie(\Gcal/C).
	\end{equation*}
	\end{prop}
	\begin{proof}
	Proposition~\ref{prop:lissitequotientsfppf} ensures that $\Bcal'$ is also smooth over $C$. 
	Hodge bundles being dual to the corresponding Lie algebras (Proposition~\ref{prop:duality.hodge.lie}),  it is  enough to prove that
	$\deg \Lie(\Bcal'/C) = \deg \Lie(\Acal/C) + (p-1) \deg \Lie(\Gcal/C)$.

	Proposition~\ref{prop:thmprincipalYuan}, together with additivity of the degree in exact sequences of line bundles (Lemma~\ref{lemexactsequenceLie}),  yields
	\[ \deg \Lie(\Gcal/C) - \deg \Lie(\Acal/C) + \deg \Lie(\Bcal'/C) - \deg\left(\Frob_S^* \Lie(\Gcal/C)\right) 
	= 0. \]
	Proposition~\ref{prop:Hodgebundleofptwist} asserts that $\deg\left(\Frob_C^*\Lie(\Gcal/C)\right) = p\, \deg \Lie(\Gcal/C)$. 
	Combining the last two equalities gives the conclusion.
 	\end{proof}

While crucial to our argument, the above result does not conclude the proof of Theorem~\ref{itheo:hdiff} for two main reasons.
First, the degree of $\Lie(\Gcal/C)$ which appears in \eqref{eq:yuan} is difficult to control (its sign is already not obvious), even if one adds constraints on $\Gcal$.
Secondly, equality \eqref{eq:yuan} is not exactly what we are looking for: in Theorem~\ref{itheo:hdiff}, one starts with an isogeny $\phi:A \rightarrow B$ between abelian varieties over $K$, and one wishes to compare their differential height \ie{}, compare degrees of line bundles coming from the Néron models of $A$ and $B$. However, the Néron model of $B$ is not necessarily the quotient group scheme $\Acal/\Gcal$ as above. 
To overcome these difficulties, we make additional hypotheses: in the remainder of this section we focus on FV-isogenies (and prove \itemref{item:hdiff.b} in Theorem~\ref{itheo:hdiff}), in the next section we add an ordinarity condition (in order to prove  \itemref{item:hdiff.c} in Theorem~\ref{itheo:hdiff})

\subsection{N\'eron models and quotient by subgroups}
Let $\Gcal$ be a smooth group scheme over $S$.
Recall from \cite[Exposé VI$_B$, Théorème 3.10]{SGA3} that the {\it neutral component} $\Gcal^\circ$ of $\Gcal$ is the open smooth subgroup scheme such that, for every $s \in S$, $(\Gcal^\circ)_s$ is the neutral component of $\Gcal_s$.

 	\begin{prop}\label{prop:isomneutralcomponents}
	Let $A$ be a semi-stable abelian variety over $K$, and $G$ be a finite subgroup scheme of $A$. 
	Let $\phi : A\to A/G =: B$ denote the corresponding quotient isogeny.
	Let $\Acal\to C$ be the N\'eron model of $A$, and $\Bcal\to C$ be the N\'eron model of the quotient abelian variety $B$.
	Let $\Gcal$ denote the Zariski closure of $G\subset A$ in $\Acal$, and $\Bcal'$ denote the \FPPF{} quotient of $\Acal$ by $\Gcal$. 
	
	The identity map $\id_B$ induces a morphism of group schemes $\Psi : \Bcal' \rightarrow \Bcal$, 
	which restricts to an isomorphism  between the neutral components $\Psi^\circ : (\Bcal')^\circ \stackrel{\sim}{\longrightarrow} (\Bcal)^\circ$. 
	\end{prop}
 	\begin{proof} 
	First, note that $\Bcal'$ is smooth (by Proposition~\ref{prop:lissitequotientsfppf}) and that the generic fiber of $\Bcal'$ is $B$.
	The Néron mapping property then allows to uniquely extend the isogeny $\phi:A\to B$ (resp. $\id_B:B\to B$) to a group scheme morphism $\Phi : \Acal \rightarrow \Bcal$ (resp. $\Psi:\Bcal' \rightarrow \Bcal$). 
	We denote the quotient morphism by $\pi : \Acal \rightarrow \Acal/\Gcal=\Bcal'$. 
	These morphisms fit in a commutative triangle:
	\[ \xymatrix{
	\Acal \ar[rr]^{\Phi}  \ar[dr]_{\pi} & &  \Bcal\,. \\
	& \Bcal' \ar[ur]_{\Psi} & }\]
	Since $\Acal$ is semi-abelian by assumption, $\Phi$ is finite and fiber-by-fiber surjective on neutral components (see Proposition~\ref{prop:semistableisogenies}{\it (b)}).
	For dimension reasons, this implies that $\Psi_s:\Bcal'_s\rightarrow \Bcal_s$ is quasi-finite for each $s\in C$.
	So, for each $s \in C$, the restriction  $\Psi_s^\circ : (\Bcal'_s)^\circ \rightarrow (\Bcal_s)^\circ$ is surjective.

 	On the other hand, for any closed point $s \in C$, the ring $R := \Ocal_{C,s}$ is a DVR, so by  Zariski's main theorem (see \cite[\S2.3, Theorem 2']{BLR}) applied to $X=\Bcal'_{R}$, $Y=\Bcal_R$ and above the generic fiber (which is dense open in $\Spec R$), we obtain that $\Psi_s : \Bcal'_s \to \Bcal_s$ is an open immersion. Its restriction $\Psi_s^\circ : (\Bcal')_s^\circ \rightarrow \Bcal_s^\circ$ is also an open immersion.

 	Thus, for every closed point $s \in C$, the map $\Psi_s^\circ :(\Bcal')^\circ \to (\Bcal)^\circ$ is a surjective open immersion, hence an isomorphism. 
	\end{proof}

	\begin{exple}\label{exampleNonSS}
	Let $E$  be the elliptic curve defined over $\F_p(t)$  by the affine equation $E: y^2 = x^3 + t$. This is an non constant but isotrivial elliptic curve (with $j$-invariant $0$). 
    Applying Tate's algorithm, one shows that $E$ has additive reduction of type $\mathrm{II}$ at (the place corresponding to) $t=0$. The model $y^2=x^3+t$ is minimal at this place. One also checks that $E$ has additive reduction of type $\mathrm{II}^\ast$ at the place $\infty$. The minimal discriminant divisor of $E$ is therefore $\Delta(E) = 2\,(0)+10\,(\infty)$, this shows that $\hdiff(E/K) =(\deg\Delta(E))/12= 1$.

    By definition, the Frobenius twist $E^{(p)}$ of $E$ is the elliptic curve given by the model $y^2=x^3+t^p$. Let $b\in\{1, 5\}$ be the integer such that $p\equiv b\bmod{6}$. 
    One shows that $E^{(p)}$ has additive reduction at~$0$ and $\infty$: the fiber at $0$ has type $\mathrm{II}$ (resp. type $\mathrm{II}^\ast$) if $b=1$ (resp. $b=5$); the fiber at $\infty$ has type $\mathrm{II}^\ast$ (resp. type $\mathrm{II}$) if $b=1$ (resp. $b=5$). The minimal discriminant divisor of $E^{(p)}$ is then $\Delta(E^{(p)}) = 2b\,(0)+2(6-b)\,(\infty)$, and thus $\hdiff(E^{(p)}/K)=1$.

	Let $R_0$ denote the local ring of $\F_p(t)$ at $0$, and let $\Ecal\to\Spec R_0$ denote the N\'eron model of $E$ above $R_0$.
	If $b=5$, the above computation of the reduction types of $E$ and $E^{(p)}$ at $0$ (type $\mathrm{II}$ and $\mathrm{II}^\ast$ respectively) shows that $(\Ecal)^{(p)}$ is not isomorphic to the N\'eron model of $E^{(p)}$! 

    We now focus on the case where $b=1$, and we set $k=(p-1)/6$. 
    By the previous paragraph, $\Ecal$ is given by  
    \[ \{ ([x:y:z], t) : y^2 z - x^3 - t z^3 = 0 \} \smallsetminus\{([0:0:1],0)\} \subset \P^2_{\F_p}\times\mathbb{A}^1_{\F_p}.\]
    By the N\'eron property, the Frobenius isogeny $\Frob_{E/\F_p(t)}:E\to E^{(p)}\simeq E$ induces a map $F : \Ecal\to\Ecal$  given by 
	\[F:([x:y:z], t) \mapsto ([t^k x^p :y^p : t^{3k} z^p], t).\]    
	On the special fiber of $\Ecal$ (which is irreducible), this map induces $[x:y:z]\mapsto [0:y:0]$, and this is obviously not an isomorphism. 
  	This last remark proves that the semi-stability hypothesis  in Proposition~\ref{prop:isomneutralcomponents} cannot be relaxed. 
	\end{exple}

The above proposition directly implies:	
	\begin{coro}\label{cor:LieA/G} 
	Let $A/K$ be a semi-stable abelian variety, and $G$ be a finite subgroup scheme of $A$. 
	Denote  the N\'eron model of $A$ by $\Acal\to C$, and the Zariski closure of $G$ in $\Acal$ by $\Gcal$. Let $\Bcal'$ denote the \FPPF{} quotient $\Acal/\Gcal$.
 	Let $B$ denote the quotient abelian variety $A/G$, and let $\Bcal\to C$ denote its N\'eron model. 
 	There is an isomorphism of Lie algebras:
	\[ \Lie(\Bcal'/C) \simeq \Lie(\Bcal/C).\]
	Hence there is an isomorphism $\Omega(\Bcal'/C) \simeq \Omega(\Bcal/C)$.
	In particular, we have $\deg\Omega(\Bcal'/C) =  \deg\Omega(\Bcal/C)$.
	\end{coro}
	
	\begin{rema}
 	Keep the same notation as in the above corollary. 
 	If $A$ is not assumed to be semi-stable, one can still relate $\deg\Omega(\Acal/\Gcal)$ to $\deg\Omega(\Bcal)$ albeit in a weaker form: Theorem 2.6 in \cite{Yua} indeed shows that
	\[\deg\Omega(\Bcal'/C) \geq  \deg\Omega(\Bcal/C).\]
	(see also Lemma~\ref{lem:keyinequalityparallelogram} for a detailed proof).
	\end{rema}

We can now state the main result of this paragraph, which will be crucial in the rest of this work:
	\begin{prop}\label{prop:varheightsemistable}
	Let $A/K$ be a semi-stable abelian variety and let $G \subset A$ be a finite subgroup scheme of Frobenius height $\leq 1$ with Zariski closure $\Gcal$ in the Néron model of $A$ over $C$. We then have      
    \[\hdiff((A/G)/K)=\hdiff(A/K) - (p-1)\,\deg\Lie(\Gcal/C).\]
    \end{prop}
	\begin{proof} 
	By definition of the differential height, this  follows from Proposition~\ref{prop:Yuan}, Proposition~\ref{prop:isomneutralcomponents}, and Corollary~\ref{cor:LieA/G}.
	\end{proof}
 
	\begin{rema}
	Lemma 4.12 in \cite{R20} proves that the differential height is constant through an isogeny $\mathcal{A}\to\mathcal{A}/\mathcal{G}$ whose kernel is a finite flat closed subgroup scheme $\mathcal{G}$ which is multiplicative (hence diagonalisable by \cite[Lemma 4.11]{R20}). 
	\end{rema}

\subsection{Variation of height by Frobenius and Verschiebung}
Using the results gathered this far, we are now able to tackle one further case of Theorem~\ref{itheo:hdiff}. 
We first describe the variation of differential height through the Frobenius isogeny:
	\begin{prop}\label{hdiff_by_Frob}
    Let $K$ be a function field of positive characteristic $p$.
	For any abelian variety $A$ over $K$, we have
	\[\hdiff(A^{(p)}/K) \leq p\,\hdiff(A/K).\]
	Furthermore, if $A$ is semi-stable over $K$, this is an equality.
	\end{prop}

This result is a direct generalization of Lemma 5.6 in \cite{GriPaz}.
Just as in the case of elliptic curves, note that the inequality may be strict without the semi-stability assumption:  it is possible for $\hdiff(A/K)$ and $\hdiff(A^{(p)}/K)$ to be equal, as shown in Example~\ref{exampleNonSS}.

	\begin{proof}
	Let $A$ be an abelian variety over $K$, and  $\Acal$, $\Bcal$ be  the Néron models of $A$ and $A^{(p)}$ over $C$, respectively.
	The identity morphism $\mathrm{id}_{A^{(p)}} : A^{(p)}\rightarrow A^{(p)}$ extends by the Néron property into a morphism 
	$\Acal^{(p)} \rightarrow \Bcal$ which, in turn, induces a morphism of Lie algebras
	$i:\Lie(\Acal^{(p)}/C) \rightarrow \Lie(\Bcal/C)$.
	This morphism is an isomorphism on the generic fibers so, by Lemma~\ref{lem:keyinequalityparallelogram}, we  have  $\deg \Lie(\Acal^{(p)}/C) \leq \deg \Lie(\Bcal/C)$. 
	Proposition~\ref{prop:Hodgebundleofptwist} then yields 
	\[ - p\,  \hdiff(A/K) \leq - \hdiff(A^{(p)}/K),\]
	which is the desired inequality. 

	Assume now that $A$ is semi-stable over $K$. Lemma~\ref{lemm:faithfulflatnessFrobeniussmoothgroups} then asserts that the Frobenius morphism $\Frob_{\Acal/C} : \Acal \rightarrow \Acal^{(p)}$ is faithfully flat.
	Its kernel $\Gcal = \Acal[\Frob_{\Acal/C}]$ is a finite flat group scheme over $C$.
	Consequently, $\Acal^{(p)}$ identifies with the \FPPF{} quotient $\Acal/\Acal[\Frob_{\Acal/C}]$.
	Moreover, as $A$ is semi-stable,  Proposition~\ref{prop:isomneutralcomponents} ensures that the neutral component of $\Acal/\Acal[\Frob_{\Acal/C}]$ is isomorphic to $\Bcal^\circ$. 
	In particular, $i$ is then an isomorphism of Lie algebras : $\Lie(\Acal^{(p)}/C) \simeq \Lie(\Bcal/C)$. 
	The equality follows by taking degrees. 

	Alternatively, one could also apply Proposition~\ref{prop:isomneutralcomponents} to the subgroup scheme $G = A[\Frob_{A/K}]$. 
	The Zariski closure of~$G$ in $\Acal$ is then $\Gcal = \Acal[\Frob_{\Acal/C}]$, as we have shown, because
    \[\deg \Lie(\Gcal/C) = \deg \Lie (\Acal[\Frob_{\Acal/C}]/C) = \deg \Lie(\Acal/C)= - \hdiff(A/K), \]
    by taking degrees in Proposition~\ref{prop:lie.ker.frob}. 
	\end{proof}

The above discussion allows to prove assertion \itemref{item:hdiff.b} of Theorem~\ref{itheo:hdiff}:
 
	\begin{coro}\label{VerFrob}
    Let $K=k(C)$ be a function field of positive characteristic $p$, and $e,f\geq 0$ be integers.
	Let $\phi:A \to B$ be a FV-isogeny of type $(e,f)$ between semi-stable abelian varieties over $K$.
    Then we have
		\[\hdiff(B/K) = p^{e-f} \, \hdiff(A/K).\]
	\end{coro}
	\begin{proof}
	Apply repeatedly Proposition~\ref{hdiff_by_Frob} and Theorem~\ref{theo:diffheight.bisep.isogenies}, the identity ensues.
	\end{proof}
    
The previous corollary is a direct generalization of Proposition 5.8 of \cite{GriPaz} to higher dimensional abelian varieties.
This result implies that, given a semi-stable abelian variety $A$ over $K$, we have
    \[\big\{\hdiff(B/K) \text{ s.t. there exists a FV-isogeny }A\to B\big\}  =p^{\Z}\,\hdiff(A/K).\]
However, in contrast to the dimension $1$ case, not all isogenies are FV-isogenies in higher dimensions, as can be seen in the following simple example.

	\begin{exple}\label{exple:not.all.FV}
	Choose a non-isotrivial semi-stable elliptic curve $E$ defined over the subfield $K^p$ of $K$ and an integer $g\geq 2$, and set $A=E^g$.
	We then have $\hdiff(A/K) = g\, \hdiff(E/K)$.
	For any integers $i,j\in\{0, \dots, g\}$ such that $i+j\leq g$, define the product $A_{i,j} := (E^{(p)})^i\times (E^{(1/p)})^{j}\times E^{g-i-j}$.
	Consider the isogeny $\phi_{i,j} : A\to A_{i,j}$  given by the Frobenius $\Frob_{E/K} : E\to E^{(p)}$ on the first $i$ factors of $A$, by the Verschiebung $\Ver_{E/K} : E\to E^{(1/p)}$ on the next $j$ factors, and  by $\id_E:E\to E$ on the remaining factors. 
	Among the various isogenies $\phi_{i,j}$, the only FV-isogenies are $\phi_{0,0}$, $\phi_{0,g}$, and $\phi_{g,0}$ (which are of type $(0,0)$, $(0,1)$, and $(1,0)$ respectively).

	For any choices of $i,j$ as above, it is clear that $\deg\phi_{i,j} = (\deg\Frob_{E/K})^i\, (\deg\Ver_{E/K})^j = p^{i+j}$. 
	Given the relation between the heights of a semi-stable abelian variety and its Frobenius/Verschiebung twists, the target abelian variety $A_{i,j}$ of $\phi_{i,j}$ has differential height
	\begin{align*}
    \hdiff(A_{i,j}/K) &= i\, p\, \hdiff(E/K) + j\, p^{-1}\, \hdiff(E/K) +  (g-i-j)\, \hdiff(E/K)\\
	 &=  \hdiff(A/K) + \left((p-1)i + (p^{-1}-1)j\right)\hdiff(E/K).
	\end{align*}
	It follows that ${\hdiff(A_{i,j}/K)}/{\hdiff(A/K)} = 1+\frac{i}{g}(p-1) + \frac{j}{g}(p^{-1}-1)$. 
	If the isogeny $\phi_{i,j}$ is not a FV-isogeny (\ie{}, if $(i,j)\notin\{(0,0), (0,g), (g,0)\}$), the ratio ${\hdiff(A_{i,j}/K)}/{\hdiff(A/K)}$ is   not a power of $p$.

	Note that, for any integers $i,j$ as above, we have
	\[ p^{-1}\,\hdiff(A/K) \leq \hdiff(A_{i,j}/K) \leq p\,\hdiff(A/K), \]
	(the extreme cases are obtained for $(i,j)=(0,g)$ and $(g, 0)$). 
	One checks that $\delta_p(\phi_{i,j})$ and $\delta_{p}(\widehat{\phi_{i,j}})$ are both $\leq 1$, so the above inequalities are consistent with the third item in our Theorem~\ref{itheo:hdiff}. 
	\end{exple}

     Proposition~\ref{hdiff_by_Frob} also allows to describe the variation of modular height (see paragraph~\ref{modular setting} for more details on the context) through the Frobenius isogeny: 

	\begin{prop} 
    Let $K'/K$ be a finite extension.
    For any point $j(A,\xi,\nu)\in{\mathscr{A}}_{g,d,n}({K'})$ induced  by a semi-stable abelian variety $A/K'$ equipped with a polarization $\lambda$ and level structure $\xi$, we let
        $j(A^{(p)},\xi^{(p)},\nu^{(p)})\in{\mathscr{A}_{g,d,n}}(K')$ denote the moduli attached to the $p$-th Frobenius twist $A^{(p)}$ of $A$ equipped with the naturally associated polarization $\lambda^{(p)}$ and level structure $\xi^{(p)}$.
		We then have
		\[\hmod(A^{(p)}, \lambda^{(p)}, \xi^{(p)}) = p \, \hmod(A, \lambda, \xi). \]
	\end{prop}
	\begin{proof} 
 By application of the Formule Cl\'e~\ref{theo-clef}, the modular height is proportional to the stable differential height, hence the claimed equality may be deduced from Proposition~\ref{hdiff_by_Frob}.	
	\end{proof}

\section{End of the proof of Theorem~\ref{itheo:hdiff}}\label{sec:proof of theorem}

In this section we gather the various ingredients introduced in the last two sections in order to finish the proof or our main theorem. Theorem~\ref{theo:diffheight.bisep.isogenies} and Corollary~\ref{VerFrob} already prove parts \itemref{item:hdiff.a} and \itemref{item:hdiff.b} of Theorem~\ref{itheo:hdiff}, only the third assertion remains to be treated in the case of a positive characteristic base field.

For convenience, let us recall the statement we are aiming at.

	\begin{theo}\label{theo:hdiff.general}
	Let $K$ be a function field of positive characteristic $p$. 
	Let $\phi : A\to B$ be an isogeny between two semi-stable ordinary abelian varieties defined over $K$. Then
	\[p^{-\delta_p(\widehat{\phi})}\,\hdiff(A/K)\leq \hdiff(B/K) \leq p^{\delta_p(\phi)}\,\hdiff(A/K) \,,\]
	where $\delta_p(-)$ denotes the Frobenius height of the kernel of the isogeny.
	\end{theo}

In the special case where $\dim A=\dim B= 1$, Lemma~\ref{lemm:frob.height.dim1} allows to rewrite the above inequalities in the form:
\[(\deg\ins(\widehat{\phi}))^{-1}\,\hdiff(A/K) \leq \hdiff(B/K) \leq \deg\ins({\phi})\,\hdiff(A/K).\]
This should be compared to the semi-stable case of Proposition 5.8 in \cite{GriPaz}, which gives the more precise equality
\[\hdiff(B/K) = \frac{\degins(\phi)}{\degins(\widehat{\phi})}\, \hdiff(A/K).\]

\subsection{Isogenies of finite Frobenius height} 
Recall the following positivity results for  Hodge bundles of abelian varieties over $K$.
	\begin{prop}\label{prop:nefnessHodgebundle}
    Let $A$ be an abelian variety of dimension $g$ over $K$. Then the line bundle  $\bigwedge^g\Omega_A$ is nef.
    
    Moreover, if  $A$ is  ordinary then the Hodge bundle $\Omega_A$ itself is nef, therefore all its quotients are also nef.
	\end{prop}
	\begin{proof}
    The line bundle $\bigwedge^g\Omega_A$ has non negative degree by Corollary~\ref{cor:MoretBailly}, hence is nef. 
    The nefness of $\Omega_A$ in the ordinary case has been proven by Rössler in \cite[Theorem 1.2]{Rossler15}, and nefness of vector bundles is stable by quotient \cite[Theorem 6.2.12]{Lazarsfeld04}.
	 \end{proof}

The nefness property of the Hodge bundle \emph{in the ordinary case} implies the following, which is the final key ingredient in our argument. 
	
	\begin{prop}\label{prop:keyinequalityordinary}
		Let $A$ be a semi-stable ordinary abelian variety. Let $G$ be a finite subgroup scheme of $A$ with $\delta_p(G)\leq 1$. Let $\phi:A\to B:=A/G$ be the corresponding isogeny. We have
  		\[ \hdiff(A/K) \leq \hdiff(B/K) \leq p\, \hdiff(A/K).\] 
	\end{prop}

Before starting the proof, let us remark that this inequality is best possible in general: 
consider indeed the subgroup scheme $G=A[\Frob_{A/K}]$ with Frobenius height $1$, in which case $B = A/A[\Frob_{A/K}]$ is isomorphic to $A^{(p)}$ and so has height $\hdiff(B/K)=p\,\hdiff(A/K)$ by Proposition~\ref{hdiff_by_Frob}. 

	\begin{proof}
	Let $\Acal $ and $\Bcal$ be the respective Néron models of $A$ and $B$ over $C$, 
	let $\Gcal$ denote the Zariski closure of $G$ in $\Acal$ (note that $\Gcal \subset \Acal[\Frob_{\Acal/C}]$). 
	We let $\Qcal :=\Acal[\Frob_{\Acal/C}]/\Gcal$. 
	By Proposition~\ref{prop:varheightsemistable}, the desired inequality is  equivalent to 
	\begin{equation}\label{eq:keyineq}
    0 \leq - \deg \Lie(\Gcal/C) \leq \hdiff(A/K).
	\end{equation}		
	We now set out to prove this. 
	By Definition~\ref{defi:exact.sequence.groupschemes}, the sequence  
	\[ 0 \longrightarrow \Gcal \longrightarrow \Acal[\Frob_{\Acal/C}] \longrightarrow \Qcal \longrightarrow 0 \]
	of finite flat group schemes of Frobenius height $\leq 1$ over $C$ is exact. 
	Using the exactness of the Lie functor for those sequences (Proposition~\ref{prop:exactness.lie.functor}) and dualizing (Proposition~\ref{prop:duality.hodge.lie}), we obtain an exact sequence of Hodge bundles
	\begin{equation}\label{eq:keyexactsequence}
	0 \longrightarrow \Omega(\Qcal/C) \longrightarrow \Omega(\Acal[\Frob_{\Acal/C}]/C) \longrightarrow \Omega(\Gcal/C) \longrightarrow 0.
	\end{equation}
	The Hodge bundle of $\Gcal$ is thus a quotient of $\Omega(\Acal/C)$: as such, it has non negative degree by Proposition~\ref{prop:nefnessHodgebundle}. This shows that  $- \deg \Lie(\Gcal/C) \geq 0$, which is the left-most inequality in \eqref{eq:keyineq}. 		

	The neutral components of $\Acal/\Gcal$ and $\Bcal$ are canonically isomorphic (by Proposition~\ref{prop:isomneutralcomponents}) because $\Acal$ is semi-abelian.
	Moreover, both $\Acal[\Frob_{\Acal/C}]$ and $\Qcal$ are connected group schemes, so that $\Acal[\Frob_{\Acal/C}]$ is contained in the neutral component $\Acal^\circ$ of $\Acal$, and $\Qcal$ canonically identifies with a closed subgroup scheme of $\Bcal^\circ$ (in particular, of $\Bcal$).
	The abelian variety $B$ is ordinary (because it is isogenous to $A$, which is ordinary). 
	In particular, the Hodge bundle $\Omega (\Bcal/C)$ is nef (see Proposition~\ref{prop:nefnessHodgebundle}), hence its quotient $\Omega(\Qcal/C)$ also is.
	We thus have  $-\deg\Lie(\Qcal/C) = \deg\Omega(\Qcal/C)\geq 0$. 

	Using the exact sequence \eqref{eq:keyexactsequence} of Lie algebras, together with additivity of degrees (Lemma~\ref{lemm:suiteexactefibresvect}), we obtain 
	\begin{align*}
	\hdiff(A/K) - (- \deg \Lie(\Gcal/C)) =  -\big(\deg \Lie(\Acal[\Frob_{\Acal/C}]/C) - \deg \Lie(\Gcal/C) \big) =   - \deg  \Lie(\Qcal/C)\geq 0.  
	\end{align*}
	This is the right-most inequality in \eqref{eq:keyineq}.
	\end{proof}

We now extend Proposition~\ref{prop:keyinequalityordinary} to isogenies with connected kernel.
	\begin{prop}\label{prop:hdiff.finite.depth}
	Let $A$ be an ordinary semi-stable  abelian variety over $K$, and let $G$ be a finite connected subgroup scheme of $A$. 
	Let $\psi : A\to B:=A/G$ be the corresponding quotient isogeny.
	Then 
		\[
		\hdiff(A/K) \leq \hdiff(B/K) \leq p^{\delta_p(\psi)}\, \hdiff(A/K).
		\]      
	\end{prop}
	\begin{proof}
	Since $G$ is connected, it has finite Frobenius height $d$ (see Proposition~\ref{prop:finite.frob.height}), and the isogeny $\psi$ is purely inseparable. 
	We argue by induction on $d\geq 1$ : if $d\leq 1$, we have already proved the desired bounds in the previous proposition. 
	Assume thus that the statement of Proposition~\ref{prop:hdiff.finite.depth} holds for all subgroup schemes with given Frobenius height $d\geq 1$ of semi-stable ordinary abelian varieties over $K$.

	Let $G\subset A$ be a group of finite Frobenius height $d+1$ in a semi-stable ordinary abelian variety $A$ over $K$.
	Lemma~\ref{lemm:frob.height.lowering} states that the quotient group scheme $G':=G/G[\Frob_{G/K}]$ has Frobenius height $d$. 
	It is also clear that the subgroup scheme $G[\Frob_{G/K}]$ of $G$ has Frobenius height $1$.
	Let $A'$ denote the quotient $A' := A/G[\Frob_{G/K}]$, and $B'$ denote the quotient of $A'$ by its subgroup scheme $G'$. 
	Then $B' = (A/G[\Frob_{G/K}]) / (G/G[\Frob_{G/K}])$ is isomorphic to $A/G = B$.
	The kernel $G/G[\Frob_{G/K}]$ of the isogeny $A' \to B'$ has Frobenius height $d$: we deduce from the induction hypothesis that 
	\[\hdiff(A'/K) \leq \hdiff(B'/K) \leq p^d \,\hdiff(A'/K).\]
	Applying Proposition~\ref{prop:keyinequalityordinary} to the isogeny $A\to A'$ (whose kernel $G[\Frob_{G/K}]$  has Frobenius height $1$) yields 
	\[\hdiff(A/K)\leq \hdiff(A'/K)\leq p\, \hdiff(A/K).\]
	Combining the two bounds leads to $\hdiff(A/K)\leq \hdiff(B/K)\leq p^{d+1}\, \hdiff(A/K)$, as was to be shown.
	\end{proof}

\subsection{Proof of Theorem~\ref{theo:hdiff.general}}
	With the results of the last subsection, we can now prove the main result of this section. 
	Let $\phi: A \rightarrow B$ be an isogeny between two non-isotrivial abelian varieties defined over $K$.

	Without loss of generality, we may and do assume that  the degree of $\phi$ is a power of $p$. 
	Indeed, in the factorization of $\phi$ as in  Proposition~\ref{prop:prime.power.decomposition}, all the factors except the one of degree a power of $p$ are biseparable and thus preserve the differential height (see Theorem~\ref{theo:diffheight.bisep.isogenies}{\it(a)}).
	Furthermore, the isogeny $\phi$ decomposes uniquely as $\phi\sep \circ \phi\ins$ by Proposition~\ref{prop:isogeny.factor.sep.insep} with $\phi\sep : A'\to B$ separable and $\phi\ins : A\to A'$ purely inseparable.
	The separable part $\phi\sep$ must have purely inseparable dual (since $\phi\sep$ is separable of degree  a power of $p$).  
	Notice that $\phi^\vee = (\phi\ins)^\vee \circ (\phi\sep)^\vee$, so that the kernel of the purely inseparable part of $\phi^\vee$ contains the kernel of $(\phi\sep)^\vee$. Hence $\delta_p((\phi\sep)^\vee) \leq \delta_p(\phi^\vee)$. 
	On the other hand we have $\delta_p(\phi)=\delta_p(\phi\ins)$ by definition.
	To conclude the proof of the theorem, it is thus enough to prove that 
		\[
		\hdiff(A/K) \leq  \hdiff(A'/K)  \leq p^{\delta_p(\phi\ins)}\, \hdiff(A/K),  
  \quad \text{ and that } \quad 
  \hdiff(B/K)\leq \hdiff(A'/K)\leq p^{\delta_p ((\phi_{\sep})^\vee)}\, \hdiff(B/K). 
		\]
	Applying Proposition~\ref{prop:hdiff.finite.depth} to $\psi = \phi\ins : A\to A'$ and $\psi = (\phi\sep)^\vee : B^\vee\to (A')^\vee$ yields both of the above inequalities (the semi-stability assumption ensures that $\hdiff(B^\vee/K)=\hdiff(B/K)$ and $\hdiff((A')^\vee/K)=\hdiff(A'/K)$ by Proposition~\ref{prop:hdiff.dual}).
	\hfill$\Box$
		
	\medskip
	
	The above paragraph also concludes the proof of Theorem~\ref{itheo:hdiff}.

\section{Examples and counterexamples}\label{sec:ex and cex}

The first four  paragraphs of this section present interesting cases inspired from Moret-Bailly's, Kani's, and Helm's respective works in characteristic $p>0$. 
The last paragraph gives an application of Theorem~\ref{itheo:hdiff} to the study of quotients of abelian surfaces that are products of elliptic curves.
 
\subsection{Moret-Bailly's example}\label{moretbabapara}
	
This paragraph discusses a classical construction by Moret-Bailly: we refer to \cite[pp. 128--129, p. 138]{MB81} for details and proofs. All the abelian surfaces considered below will be in the isogeny class of a constant abelian surface, in particular all semi-stable.

Choose two supersingular elliptic curves $E_1$ and $E_2$ defined over an algebraically closed field $k$ of characteristic $p>2$,  and let $A_0 = E_1 \times E_2$ be the product abelian surface, polarized by the line bundle $L_0=\mathcal{O}(E_1\times \{0_2\} + \{0_1\}\times E_2)^{\otimes p}$.
The group scheme $K(L_0)$ is then $E_1[\,p\,] \times E_2[\,p\,]$. 
There exists a subgroup scheme $G \subset K(L_0)$ which is totally isotropic with respect to the Weil pairing, and such that $G \simeq \aalpha_p$ and $G^{\perp}/G\simeq \aalpha_p\times \aalpha_p$. 
On the quotient abelian surface $A := (E_1\times E_2)/G$, the line bundle $L_0$ descends to a symmetric line bundle $L$ which satisfies $K(L)\simeq \aalpha_p\times\aalpha_p$. 
     
In this paragraph, we let $S=\mathbb{P}^1_k$.
Denote by $\Acal = A\times_{k}S$ the (constant) abelian surface over $S$ deduced from $A$, and let $\Lcal = L\times_kS$ be the corresponding line bundle on $\Acal$. 
By construction, there is an isomorphism $\iota : K(\Lcal) \stackrel{\sim}{\longrightarrow} (\aalpha_p \times \aalpha_p)_S$.
Write $(\aalpha_p \times \aalpha_p)_S $ as $(\aalpha_p \times \aalpha_p)_S = \Spec\left( \Ocal_S[\alpha,\beta]/(\alpha^p,\beta^p)\right)$.
The group scheme $K(\Lcal)$ contains a subgroup scheme $\Hcal$ over $S$ isomorphic (via $\iota$) to the subgroup of $(\aalpha_p\times\aalpha_p)_S$ given by the equation $\alpha\, X - \beta\, Y =0$. 
This finite flat subgroup scheme $\Hcal$ is not constant but, for every closed point $s \in S$, $\Hcal_s \simeq (\alpha_p)_s$.

Let $H$ denote the generic fiber of $\Hcal$. The second abelian variety of interest is  the quotient $B := A/H$, with quotient isogeny $\pi : A \rightarrow A/H$.

	\begin{lemm}\label{lemoretbaba}
	In the above situation, we have $\hdiff((A_0)_K/K) = \hdiff(A_K/K) = 0$, and $\hdiff(B/K) = (p-1)$.
	\end{lemm}
 
\begin{proof}
    Both $(A_0)_K$ and $A_K$ are constant over $K$, Corollary~\ref{cor:MoretBailly} directly yields the values of their differential height. Proposition~\ref{prop:varheightsemistable} states that
    \[\hdiff(B/K) = \hdiff(A_K/K) - (p-1) \deg \Lie(\Hcal/\P^1_k) = 0 + (p-1),    \]
    where the last equality is a consequence of \cite[\S 1.4]{MB81}, where it is proved that  $\Lie(\Hcal/\P^1_k) \simeq \Ocal_S(-1)$.  
\end{proof} 

This example where $\hdiff(B/K)> p \, \hdiff(A/K)$ visibly contradicts the right-most inequality in Proposition~\ref{prop:hdiff.finite.depth}. This shows that the ordinarity assumption cannot be relaxed for this inequality to hold. 

In fact, Moret-Bailly's construction also provides a counterexample in the non-ordinary case to the left-most inequality of Proposition~\ref{prop:hdiff.finite.depth}. 
Indeed, consider the dual isogeny $\widehat{\pi}:B^\vee \rightarrow A^\vee$. 
Its kernel is isomorphic to $\aalpha_p^\vee = \aalpha_p$ by Proposition~\ref{prop:dualisogenykernel}, hence is also connected. Now, $\hdiff(B^\vee/K) = \hdiff(B/K)=(p-1)$ and $\hdiff(A^\vee/K) = \hdiff(A/K) = 0$ by Proposition~\ref{prop:hdiff.dual}.  
Therefore $\widehat{\pi}:B^\vee \to A^\vee$ is an isogeny with connected kernel and we have $\hdiff(B^\vee/K)>\hdiff(A^\vee/K)$ \ie{}, the differential height \emph{decreases} through $\widehat{\pi}$. 
 
Other families inspired by this Moret-Bailly construction are presented in \cite{RS22}, see in particular p.\,100 there.
Generalizing  Moret-Bailly's pencil of supersingular abelian surfaces to higher dimensions, they construct, for each field of characteristic $p > 0$, a smooth projective variety with trivial dualizing sheaf that does not lift to characteristic zero.

\subsection{A generalization of Moret-Bailly's example}\label{s:moretba.steroids}
In fact, the `unusual' height variation exhibited by Moret-Bailly's isogeny is not an isolated example. As we now explain, one can produce isogenies between higher-dimensional (non-ordinary) abelian varieties for where the same behavior occurs.
Let $k$ be an algebraically closed field of characteristic $p>2$.
We start with a lemma:

    \begin{lemm}
    Let $\Gcal$ be a finite flat group scheme over $\P^1_k$ with generic fiber  $\aalpha_p \times \aalpha_p$.
    Then there exists an integer $n_0 \in \Z$ such that, for every  $n< n_0$, 
    $\Gcal$ admits a flat closed subgroup scheme $\Hcal \subset \Gcal$  of order $p$, with generic fiber $\aalpha_p$, and such that $\Lie(\Hcal/\P^1_k) \simeq \Ocal_{\P^1_k}(n)$.
    \end{lemm}

    \begin{proof}
    Recall from  \cite[Exposé VII$_A$, \S7]{SGA3} that the functor $\Hcal \mapsto \Lie(\Hcal/\P^1_k)$ is an equivalence of categories between the finite flat commutative group schemes over $\P^1_k$ of Frobenius height $\leq 1$ and the commutative locally free $p$-Lie algebras over $\P^1_k$. 
    If $\Hcal$ is such a  group scheme and has order $p^m$, then
    $\Lie(\Hcal/\P^1_k)$ is locally free of rank $m$ over~$\P^1_k$.
    By assumption, the generic fiber $\Gcal_{k(t)}$ is isomorphic to $\aalpha_p \times \aalpha_p$, so the $p$-Lie algebra structural morphism on $\Lie(\Gcal/\P^1_k)$ is trivial on its generic fiber, hence it is $0$ everywhere.
    As a consequence, every subvector bundle of $\Lie(\Gcal/\P^1_k)$ is automatically a sub-$p$-Lie algebra of $\Lie(\Gcal/\P^1_k)$  and thus, by the equivalence of categories above, corresponds to a finite flat subgroup scheme of $\Gcal$. 
    Proving the lemma then boils down to finding  a suitable subvector bundle of $\Lie(\Gcal/\P^1_k)$ of rank $1$.
    
    By the structure theorem for vector bundles on $\P^1_k$ (see e.g. \cite[Theorem 11.51]{GortzWedhorn}), there exist unique integers $n_1\geq n_2$ such that $\Lie(\Gcal/\P^1_k) \simeq \Ocal_{\P^1_k}(n_1) \oplus \Ocal_{\P^1_k}(n_2)$.
    We now show that $\Ocal_{\P^1_k}(n_1) \oplus \Ocal_{\P^1_k}(n_2)$ contains a subvector bundle isomorphic to $\Ocal_{\P^1_k}(n)$ for suitable $n\in\Z$.
    
    For any integers $m, m'$, maps of coherent sheaves $\Ocal_{\P^1_k}(m) \rightarrow \Ocal_{\P^1_k}(m')$  canonically correspond to multiplication by homogeneous polynomials of degree $m' - m$. 
    Given two integers $n_1, n_2$, let $n$ be an integer such that $n< \min(n_1,n_2)$. 
    One can find a pair of coprime homogeneous polynomials $F_1, F_2\in k[T,U]$ of respective degrees $n_1-n$ and $n_2-n$.
    One then constructs a map of vector bundles $f:\O_{\P^1_k}(n)\rightarrow \Ocal_{\P^1_k}(n_1) \oplus \Ocal_{\P^1_k}(n_2)$ by 
    $f=f_1\oplus f_2$ where each of $f_1, f_2$ is  multiplication by the corresponding uppercase-letter polynomial.
    Proposition 5 in \cite{BK01} then shows (in particular) that $f$ is a monomorphism in the category of vector bundles over $\P^1_k$.
    These computations mean  that, for any integer $n < \min(n_1,n_2)$, $\Lie(\Gcal/\P^1_k)$ contains $\Ocal_{\P^1_k}(n)$ as a subvector bundle. 
    
    Therefore $\Gcal$ contains a closed flat subgroup scheme $\Hcal$ with $\Lie(\Hcal/\P^1_k) \simeq \Ocal_{\P^1_k}(n)$, as claimed.
    The generic fiber~$H$ of $\Hcal$ is a subgroup scheme of order $p$ in  $\Gcal_{k(t)}\simeq \aalpha_p\times\aalpha_p$. The $p$-Lie algebra of $H$ (over $k(t)$) is commutative, has dimension $1$ and  trivial $p$-Lie algebra structural morphism.  Therefore $H$ must be isomorphic to $\aalpha_p$ over $k(t)$.
    \end{proof}

From the above lemma, we deduce
    \begin{prop}\label{prop:moretba.gen}
    Let $A$ be a semi-stable abelian variety over $K=k(t)$ such that $A[\,p\,]$ contains a copy of $\aalpha_p$.
    Then, there exists an integer $d_0\in\Z$ such that, for all $d \geq d_0$, the square variety $B = A \times A$ admits an isogeny $\phi : B \rightarrow B_d$ of Frobenius height $\leq 1$ and of degree $p$ (whose dual shares these two properties) with 
    \[ \hdiff(B_d/K) - \hdiff(B/K) = d\, (p-1). \]
    \end{prop}
    \begin{proof}
    Let $A$ be as in the statement, let $G_0\subset A$ be a subgroup scheme of $A$ with $G_0\simeq \aalpha_p$ over $K$.
    Let $\Bcal$ be the Néron model of $B=A\times A$ over $\P^1_k$, and let $\Gcal$ be the Zariski closure of $G_0\times G_0$ in $\Bcal$. 
    By the previous lemma  there exists an integer $d_0$ such that for any $d\geq d_0$, $\Gcal$ contains a subgroup scheme $\Hcal_d$ with generic fiber  isomorphic to $\aalpha_p$ and $\Lie(\Hcal_d/\P^1_k)\simeq\O_{\P^1_k}(-d)$. In particular, we have
    $\deg\Lie(\Hcal_d/\P^1_k) = -d$.
        
    Let $H_d\subset B$ denote the generic fiber of $\Hcal_d$, and let $\phi_d:B\to B/H_d=:B_d$ be the quotient isogeny. 
    By construction, the kernel $H_d$ of $\phi_d$ is isomorphic to $\aalpha_p$. Hence $\phi_d$ has degree $p$ and Frobenius height $1$, and the same holds for the dual of $\phi_d$ because $\aalpha_p$ is Cartier autodual. 
    The Zariski closure of $\ker\phi_d$ in $\Bcal$ is $\Hcal_d$, so Proposition~\ref{prop:varheightsemistable} yields 
     the desired relation between $\hdiff(B_d/K)$ and $\hdiff(B/K)$.
     \end{proof}

This proposition allows one to reproduc,  in any even dimension at least $4$ and with no isotriviality condition, the height variation observed in section \ref{moretbabapara} in Moret-Bailly's example. 
Indeed, let $B = A^2$ be an abelian variety as in Proposition \ref{prop:moretba.gen} and let $d$ be a large enough integer. 
Then there is a purely inseparable isogeny $\phi_d : B\to B_d$ of degree~$p$ with  $\delta_p(\phi)=\delta_p(\widehat{\phi})=1$, 
even though
$\hdiff(B_d/K) < \hdiff(B/K)$.
This obviously violates the inequalities of Theorem~\ref{theo:hdiff.general}.  
This shows that the ordinarity assumption in that theorem cannot be lifted.

\subsection{Kani's construction}\label{Kani}

Assume that $K$ has odd characteristic $p$. Let $E$ be a semi-stable non-isotrivial elliptic curve over $K$.
Choose a point $P$ on $E$ of exact order $m:=(p^2-1)/4$, and let $P_1:=[\frac{p+1}{2}]P$ and $P_2:=[\frac{p-1}{2}]P$.
Up to a finite extension of $K$, we may assume that these three points are $K$-rational.
Let $H$ (resp. $H_1$, $H_2$) denote the subgroups of $E$ generated by $P$ (resp. $P_1$, $P_2$).

Let $E':=E/H$ be the quotient of $E$ by $H$ and let $\pi : E\to E'$ denote the corresponding isogeny.
Similarly, let $\pi_i : E\to E_i := E/H_i$ denote the quotient isogeny from $E$ to its quotient by $H_i$ (for $i=1, 2$).
It is then clear that $\deg\pi = (p^2-1)/4$, $\deg\pi_1 = (p-1)/2$ and $\deg\pi_2 = (p+1)/2$. 
These three isogenies $\pi, \pi_1, \pi_2$ are biseparable  since their degrees are coprime to $p$ (see \cite[Lemma 4.5]{GriPaz}).
Theorem 5.1 in \cite{GriPaz} yields 
\[ \hdiff(E/K)=\hdiff(E'/K) = \hdiff(E_1/K) = \hdiff(E_2/K).\]
For $i=1, 2$, the isogeny $\pi$ factors as $\pi=\varpi_i \circ \pi_i$ for some isogeny $\varpi_i:E_i\to E'$ since $\ker\pi_i=H_i\subset H = \ker \pi$. 
We then let $\psi : E\times E'\to E_1\times E_2$ be the isogeny given by the matrix $\begin{pmatrix} \pi_1 & -\widehat{\varpi_1} \\ \pi_2 & \widehat{\varpi_2}\end{pmatrix}$. Kani's work \cite[\S2]{Kani97} shows that the isogeny $\psi$ has degree $p^2$. 
In particular, $\psi$ is not biseparable. Furthermore, $\psi$ is minimal (as $E$ is non-CM) because $\gcd((p+1)/2,(p-1)/2))=1$, so $\ker \psi$ is $p^2$-primitive by \cite[Corollary 2.11]{Kani97}.

However, $\hdiff(E\times E') = \hdiff(E_1\times E_2) = 2\, \hdiff(E/K)$. 
This gives a non trivial example of a non biseparable cyclic isogeny between two semi-stable abelian surfaces that preserves the differential height!

\subsection{Helm's example}\label{Helm}
In \cite{Hel}, David Helm proves the following: For each prime $p$, there exists a function field $K$ of characteristic $p$ and an ordinary abelian variety $A$ over $K$ with no isotrivial factors, such that $A$ admits an \'etale isogeny $\phi : A\to A$ of degree a power of $p$. 

This gives other examples of non biseparable isogenies that preserve the differential height.

\subsection{Picard rank and abelian surfaces}
The goal of this paragraph is to explain how knowledge on isogeny classes can help understand quotients of certain abelian varieties, using previous works of Shioda on Picard ranks of algebraic surfaces. 
For a surface $X$, we denote by $\rho(X)$ the rank of its N\'eron-Severi group, this integer is called the {\it Picard rank} of $X$. 
If $E$ and $E'$ are elliptic curves, Shioda \cite{Shioda} computes the Picard rank of the product $E\times E'$:
\vspace{-0.5\topsep}\begin{itemize}\setlength{\itemsep}{0em}
    \item If $E$ and $E'$ are not isogenous, we have $\rho(E\times E')=2$.
    \item If $E$ and $E'$ are isogenous and $\mathrm{End}(E)\simeq \Z$, we have $\rho(E\times E')=3$.
    \item If $E$ and $E'$ are isogenous and $\mathrm{End}(E)\simeq \Z^2$, we have $\rho(E\times E')=4$.
    \item If $E$ and $E'$ are isogenous and $\mathrm{End}(E)\simeq \Z^4$, we have $\rho(E\times E')=6$.
\end{itemize}
We also recall Proposition 1.1 of \cite{Shioda}.
	\begin{prop}\label{Shioda}
	Let $E,E_1, E_2$ be elliptic curves such that $E\times E_1\simeq E\times E_2$. 
	If $\rho(E\times E_1)\leq3$, then  $E_1\simeq E_2$.
	\end{prop}

Shioda also shows that Proposition~\ref{Shioda} cannot be generalized without care, as shown in the following example, extracted from \cite{Shioda}. 
Let $k$ be a field of characteristic $p$ satisfying $p=3\bmod{4}$. 
Consider the elliptic curve $E/k$ given by the model $E:y^2=x^3-x$. 
Let $P_1=(0,0)\in E$ and $P_2=(1,0)\in E$. 
Each of the quotient elliptic curves $E_1:=E/\langle P_1\rangle$ and $E_2:=E/\langle P_2\rangle $ is $2$-isogenous to $E$. 
One can show that $E$ is even isomorphic to $E_1$ over $k$, and that $j(E)\neq j(E_2)$ if and only if $p\notin\{3, 7\}$. 
Assume that $p=3\bmod{4}$ and $p>7$ : the above implies that $E$ is not isomorphic to $E_2$. 
Proposition 2.3 of \cite{Shioda} shows that $E\times E_1 \simeq E\times E_2$ over $k$, despite the fact that $j(E_1)\neq j(E_2)$. 
This shows that, in general, given an abelian surface $A$ and two abelian subvarieties $B_1$ and $B_2$, the fact that $B_1$ is isomorphic to $B_2$ does not imply that $A/B_1$ is isomorphic to $A/B_2$!

We now work with elliptic curves defined over a function field $K$. 
First, we remark that Theorem A of \cite{GriPaz} gives an easy numerical criterion to prove that $\rho(E\times E')=2$: if two non-isotrivial elliptic curves $E$ and $E'$ defined over $K$ are such that the quotient $\hmod(E)/\hmod(E')$ is not a power of $p$ (in particular not equal to 1), then $\rho(E\times E')=2$. 

Exploiting this idea, we give a corollary concerning certain products of elliptic curves. 
	\begin{prop}
	Let $E_1$ and $E_2$ be two non-isotrivial elliptic curves over $K$ such that $\hmod(E_1)/\hmod(E_2)$ 
	is not a power of $p$. Then, for any elliptic curve $E/K$ with $\mathrm{End}(E)=\Z$, the surfaces $E\times E_1$ and $E\times E_2$ are not isomorphic.
	\end{prop}
	\begin{proof}
	The assumption implies, by \cite[Theorem A]{GriPaz}, that $E_1$ and $E_2$ are not isogenous. 
	In particular they are not isomorphic.
	Moreover, as $\mathrm{End}(E)=\Z$, the above mentioned work by Shioda shows that $\rho(E\times E_1)\leq 3$. 
	This implies the conclusion by Proposition~\ref{Shioda}. 
	\end{proof}

For a discussion of other illustrations of the same nature, we refer to examples 1.8 and 1.9 in \cite{R17}.
 
\section{The parallelogram inequality}\label{parallelogram}
{\it Throughout this section, the base function field $K$ is fixed, and we only consider abelian varieties defined over $K$. 
To avoid overly complicated notation, we  denote the differential height (over $K$) of an abelian variety $A$ defined over $K$ by $\hd(A)$ (instead of $\hdiff(A/K)$ as above).
}

We first recall R\'emond's theorem \cite{R22} on the Faltings heights of abelian varieties over number fields arranged in a `parallelogram' configuration.  
Let $A$ be an abelian variety defined over a number field $L$ 
and let $G, H$ be two finite subgroups of $A(\overline{L})$ defined over $L$ (\ie{}, stable under the action of the absolute Galois group $\mathrm{Gal}(\overline{L}/L)$). Then we have
\[\hfal(A/(G+H))+\hfal(A/(G\cap H)) \leq \hfal(A/G) + \hfal(A/H),\]
 where $\hfal$ denotes the Faltings height (over $L$).
 We prove the exact analogue of this statement in the context of abelian varieties over function fields: 
	\begin{theo}\label{theo:parallelogram.ineq}
	Let $K$ be the function field of a curve as above.
	Let $A$ be an abelian variety over $K$. If $K$ has positive characteristic, assume that $A$ is semi-stable. 
	For any   finite subgroup schemes $G,H$ of $A$, we have
	\begin{equation}\label{eqparallinequality}
	\hd(A/(G+H)) +\hd(A/(G\cap H)) \leq \hd(A/G)  +\hd(A/H).
	\end{equation}
	Moreover, if the order of $G$ or $H$ is not divisible by the characteristic of $K$ (e.g. if the characteristic of $K$ is $0$), then~\eqref{eqparallinequality} is an equality.
	\end{theo}

The name ``Parallelogram inequality'' stems from the commutative diagram of isogenies which one gets from the data of the theorem:
\[\begin{tikzcd}
 & A/(G\cap H) \arrow[ld] \arrow[rd] \arrow[dd, dashed] &  \\
A/G \arrow[rd] &   & A/H \arrow[ld] \\
 & A/(G+H)   &  
\end{tikzcd}.\]
Note that, in contrast to the third case of Theorem~\ref{itheo:hdiff}, Theorem~\ref{theo:parallelogram.ineq} does not add any ordinarity assumption in positive characteristic.
The proof of the parallelogram inequality \eqref{eqparallinequality} in the function field setting -- which is quite different from that of R\'emond \cite{R22} in the number field case -- occupies sections~\ref{sec:par.config},~\ref{sec:parineq.3cases}, and~\ref{sec:parineq.proof} below.

\begin{exple}
There are situations in which inequality \eqref{eqparallinequality} is strict even when $A$ is ordinary and semi-stable; here is an example.
Let $K$ be a function field of positive characteristic $p$.
Let $A$ be a semi-stable ordinary abelian variety defined over the subfield $K^p$ of $K$, and set $B=A^{(1/p)}$ (so that $A\simeq B^{(p)}$).
Let $H$ denote the kernel of the Frobenius $\Frob_{A/K} : A\to A^{(p)}$ on $A$, and $G$ denote the kernel of the Verschiebung $\Ver_{B/K} : A=B^{(p)}\to B$. 
Both  $G$ and $H$ are subgroup schemes of $A$ of order $p^{\dim A}$, and they are contained in the $p$-torsion subgroup scheme $A[\,p\,]$ (which is of order $p^{2\dim A}$).
The group scheme $H$ is connected, whereas $G$ is \'etale: their  intersection $G\cap H$ is therefore trivial. We deduce that $G+H = A[\,p\,]$. 
These remarks imply that 
\[\hd(A/(G+H)) + \hd(A/(G \cap H))=  2\, \hd(A).\]
On the other hand, Corollary~\ref{VerFrob} shows that
\[\hd(A/H) + \hd(A/G) = \hd(A^{(p)}) + \hd(B)= p\, \hd(A) + p^{-1}\, \hd(A) = (p+1/p)\,\hd(A).\]
Since $p\geq 2$, we thus have $\hd(A/(G+H)) + \hd (A/G \cap H) < \hd(A/G) + \hd(A/H)$.

Here is the reason why inequality~\eqref{eqparallinequality} is strict here (compare with the proof of Lemma~\ref{lemm:parineq.co}). 
Let $\Hcal$ denote the Zariski closure of $H$ in the Néron model of $A$, and $\Hcal'$ the Zariski closure of the image of $H$ in the Néron model of $A/G$. Applying Proposition~\ref{prop:varheightsemistable} to $A\to A/H$ and $A/G\to A/(G+H)$ yields
\begin{eqnarray*}
    \hd(A) - \hd(A/H) & = & (p-1) \deg \Lie(\Hcal/C) \\
    \hd(A/G) - \hd(A/(G+H)) & = & (p-1) \deg \Lie(\Hcal'/C). 
\end{eqnarray*}
In the current situation, $\deg \Lie(\Hcal/C) =k$ is a negative integer, and $\Hcal \cong (\Hcal')^{(p)}$ so $\deg  \Lie(\Hcal'/C)  = k/p > k$. 
Hence we have $\deg \Lie(\Hcal/C) < \deg \Lie(\Hcal'/C)$, and \eqref{eqparallinequality} is strict.
This example shows that applying an \'etale isogeny to a subgroup of Frobenius height $\leq 1$ can increase  the degree of the Lie algebra of its Zariski closure in the Néron model.
\end{exple}

As a cute consequence, we explicitly state a consequence of Theorem~\ref{theo:parallelogram.ineq} in dimension $1$ (see Corollary~\ref{itheo:parineq.ellcurves}). 
For any elliptic curve $E$ over any finite extention of $K$, we let $\hmod(E)$ denote the absolute Weil height of the $j$-invariant of $E$ (see \cite{GriPaz} for a more precise definition). 
	\begin{coro}
	Let $E$ be an elliptic curve over the algebraic closure $\bar{K}$ of $K$. Let $G,H$ be two finite subgroup schemes of $E$. Then we have 
    \[ \hmod(E/(G\cap H))+\hmod(E/(G+H))\leq \hmod(E/G) + \hmod(E/H).\]
	\end{coro}
	\begin{proof}
	Let $K'$ be a finite extension of $K$ over which $E$ attains semi-stable reduction. Corollary 3.3 in \cite{GriPaz} provides the equality $\hmod(E)=12\,\hdiff(E\times_KK'/K')$. 
	The claim thus follows from Theorem~\ref{theo:parallelogram.ineq} applied to $E\times_K K'$.
	\end{proof}

	\begin{rema}
	The parallelogram inequality \eqref{eqparallinequality} can be  proven more directly in case $A=E$ is a semi-stable non-isotrivial elliptic curve.
	Indeed, it suffices to combine the `FV-isogeny' case of Theorem~\ref{itheo:hdiff} (in dimension $1$,  all isogenies are FV-isogenies),  Proposition~\ref{lemm:frob.height.dim1}, 
	and inequalities \eqref{eq:frobh.minmax} for the Frobenius heights of sums and intersections of two finite subgroup schemes  of $E$. 
	\end{rema}

\subsection{Parallelogram configurations}\label{sec:par.config}
We now begin proving Theorem~\ref{theo:parallelogram.ineq}.
Let us consider triples $(A;G,H)$ where $A$ is an abelian variety over $K$, which is assumed to be semi-stable if $K$ has positive characteristic, and $G,H$ are finite subgroup schemes of $A$.
Such a triple $(A;G,H)$ will be called a {\it parallelogram configuration over $K$} if $G\cap H$ is the trivial group scheme. 

Let $(A;G,H)$ be a triple as above. Consider the quotient abelian variety $A':=A/G\cap H$, and its two subgroup schemes $G':=G/G\cap H$ and $H':=H/G\cap H$. The triple $(A'; G', H')$ is clearly a parallelogram configuration.
Assuming that the parallelogram inequality \eqref{eqparallinequality} holds for $(A'; G', H')$ yields
\[\hd(A') + \hd(A'/(G'+H')) \leq \hd(A'/G') + \hd(A'/H').\]
Classical isomorphism theorems imply that $A'/(G'+H') \simeq A/(G+H)$, $A'/G' \simeq A/G$, and $A'/H'\simeq A/H$. 
Hence the last displayed inequality reads
\[\hd(A/(G\cap H)) + \hd(A/(G+H)) \leq \hd(A/G) + \hd(A/H),\]
which is the parallelogram inequality for our initial triple $(A;G,H)$.
Therefore, in order to prove the parallelogram inequality \eqref{eqparallinequality} for all triples $(A;G,H)$ over $K$, it suffices to do so for all parallelogram configurations over $K$.  
\\

The next lemma will allow us to use a suitable `d\'evissage' to prove \eqref{eqparallinequality}.
\begin{lemm}\label{lemm:parineq.devissage} 
Let $(A;G,H)$ be a parallelogram configuration over $K$ (with $G\cap H$ trivial).
Let $H'$ be a subgroup scheme of $H$.
If the parallelogram inequality holds for both configurations $(A;G,H')$ and $(A/H'; (G+H')/H', H/H')$, then the parallelogram inequality holds for $(A;G,H)$.
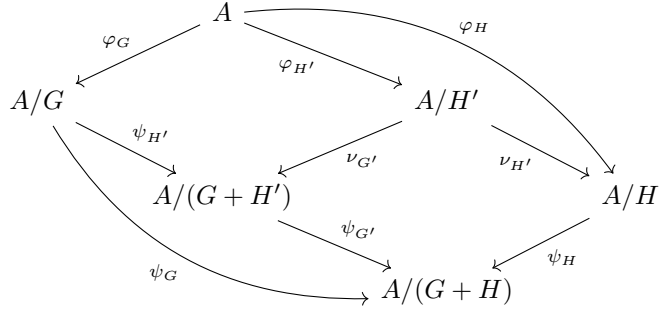
	\[
	\begin{tikzcd}
	& A \arrow[rd, "\varphi_{H'}"'] \arrow[ld, "\varphi_G"'] \arrow[rrdd, "\varphi_H", bend left] & & \\
	A/G \arrow[rd, "\psi_{H'}"] \arrow[rrdd, "\psi_G"', bend right] & & A/H' \arrow[rd, "\nu_{H'}"'] \arrow[ld, "\nu_{G'}"] & \\
	& A/(G+H') \arrow[rd, "\psi_{G'}"] & & A/H \arrow[ld, "\psi_H"] \\
	& & A/(G+H) &                         
	\end{tikzcd}\]
\end{lemm}

\begin{proof} Let $B := A/H'$ be the quotient of $A$ by $H'$, let $N := H/H'$ be the image of $H$ in $B$, and $M := (G+H')/H'$ be the image of $G$ in $B$. It is clear that $M\cap N$ is trivial (because $G\cap H$, and thus $G\cap H'$ are trivial).
By classical isomorphism theorems, we have $B/N\simeq A/H$, $B/M \simeq A/(G+H')$, and $B/(M+N) \simeq A/(G+H)$. 
Assuming that the parallelogram inequality holds for $(A;G,H')$ means that 
\[\hd(A) + \hd(A/(G+H')) \leq \hd(A/G) + \hd(A/H').\]
Rewriting the parallelogram inequality for $(B; M, N)$ with the isomorphisms above yields
\[ \hd(A/H') + \hd(A/(G+H))  \leq \hd(A/(G+H')) + \hd(A/H).\]
Summing up the two displayed inequalities  and canceling some terms finally leads to 
the desired parallelogram inequality for the configuration $(A;G,H)$.  
\end{proof}

\subsection{Special cases of Theorem~\ref{theo:parallelogram.ineq}}\label{sec:parineq.3cases}
We first prove three special cases of Theorem~\ref{theo:parallelogram.ineq} in that many lemmas.

\begin{lemm}
\label{lemm:parineq.etet}
Let $(A; G, H)$ be a parallelogram configuration over $K$. Assume that $G$ or $H$ is \'etale with \'etale dual. Then the parallelogram inequality \eqref{eqparallinequality} holds for $(A; G, H)$, and is an equality.
\end{lemm}
(Note that this lemma concludes the proof of Theorem~\ref{theo:parallelogram.ineq} in the characteristic $0$ case.)

\begin{proof} The roles of $G$, $H$ being symmetric, we may assume that $H$ is \'etale with \'etale dual. In the commutative diagram, 
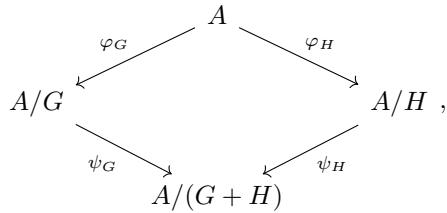
\[\begin{tikzcd}
     & A \arrow[ld, "\varphi_G"'] \arrow[rd, "\varphi_H"]   &  \\
A/G \arrow[rd, "\psi_{G}"'] &   & A/H \arrow[ld, "\psi_H"] \\
 & A/(G+H)   &  
\end{tikzcd},\]
the isogeny $\psi_G : A/G \rightarrow A/(G+H)$ has kernel $\ker \psi_G = (G+H)/G$. Since $G\cap H$ is trivial, this kernel is isomorphic to $H$, it is thus \'etale with \'etale dual. Both $\varphi_H$ and $\psi_G$ are biseparable isogenies, so that Theorem~\ref{theo:diffheight.bisep.isogenies}  yields
$\hd(A) = \hd(A/H)$, and $\hd(A/G) = \hd(A/(G+H))$.  
In particular, we have \[\hd(A)+\hd(A/((G+H)) = \hd(A/G) + \hd(A/H),\] 
which is the desired ``parallelogram equality''.
\end{proof}

\begin{lemm}
\label{lemm:parineq.co}
Let $(A; G, H)$ be a parallelogram configuration over $K$. Assume that $G$ or $H$ is connected. Then the parallelogram inequality \eqref{eqparallinequality} holds for $(A; G, H)$.
\end{lemm}
\begin{proof}  
Again, we may and do assume that $H$ is connected.
Proposition~\ref{prop:finite.frob.height} ensures that $H$ has finite Frobenius height. 
So we can filtrate $H$ by an increasing chain of subgroup schemes $H_0 = \{0 \} \subset H_1 \subset \cdots \subset H_\delta = H$, in such a way that  $H_{i+1}/H_i$ has Frobenius height $1$ for each $i \in \{0, \cdots, \delta-1\}$.
(In other words, $H_i$ is the intersection of $H$ with the kernel of the $i$-th iterated Frobenius $\Frob_{A/K}^{(i)}$). 
For any $j\in\{1, \dots, \delta+1\}$, consider the quotient abelian varieties $A_j := A/H_{j-1}$ and $A'_j := A/(G+H_{j-1})$, and let $\Acal_j$ and $\Acal'_j$ denote their respective N\'eron models over $C$.
For $j\in\{1, \dots, \delta\}$, consider the commutative diagram
\[
\begin{tikzcd}
& A_j \arrow[rd, "\varphi_j"] \arrow[ld, "\lambda_j"' ] &   \\
 A'_j \arrow[rd, "\psi_{j}"]   &  & A_{j+1}    \arrow[ld ]     \\
 & A'_{j+1}    
\end{tikzcd}.\]
The kernel of the isogeny $\varphi_j : A_j \to A_{j+1}$ is  
$\ker\varphi_j \simeq H_j/H_{j-1}$, and the isogeny $\psi_j : A'_j \to A'_{j+1}$ has kernel $\ker\psi_j = (G+H_j)/(G+H_{j-1}) \simeq H_j/H_{j-1}$. 
Both $\ker\varphi_j$ and $\ker\psi_j$ thus have Frobenius height $1$.   
Write $\Hcal_j$ for the Zariski closure of $\ker\varphi_j$  in $\Acal_j$, and $\Hcal'_j$ for that of $\ker\psi_j$ in $\Acal'_j$. Note that both $\Hcal_j$  and $\Hcal'_j$ have Frobenius height $\leq1$.

Consecutive applications of Proposition~\ref{prop:varheightsemistable} yield
\begin{align*}
    \hd(A) - \hd(A/H) &= \sum_{j=1}^\delta \big( \hd(A_j) -\hd(A_{j+1}) \big)= (p-1)\,\sum_{j=1}^\delta\deg \Lie ({\Hcal_j}/C), \\ 
\text{ and }\qquad \hd(A/G) - \hd(A/(G+H)) &= \sum_{j=1}^\delta \big( \hd(A'_j) -\hd(A'_{j+1}) \big)= (p-1)\,\sum_{j=1}^\delta\deg \Lie ( {\Hcal'_j}/C).
\end{align*}

On the other hand, for a given $j\in\{1, \dots, \delta\}$, the canonical quotient map $\lambda_j : A_j\to A'_j$ induces an isomorphism of group schemes $\ker\varphi_j\rightarrow\ker\psi_j$ (because  $G\cap H$ is trivial). 
Extending the map $\lambda_j$ by the N\'eron mapping property, and passing to the Zariski closure, this isomorphism extends into a morphism of group schemes ${\Hcal_j} \rightarrow {\Hcal'_j}$, which is an isomorphism on the generic fiber.
Lemma~\ref{lem:keyinequalityparallelogram} now implies that $\deg  \Lie({\Hcal_j}/C) \leq \deg \Lie({\Hcal'_j}/C)$ for all $j \in \{1, \cdots, \delta\}$.

With the previous paragraph, this leads to
\[
\hd(A) - \hd(A/H) \leq \hd(A/G) - \hd(A/(G+H)),
\]
which is equivalent to the parallelogram inequality for the triple $(A; G, H)$.
\end{proof}

\begin{lemm}
\label{lemm:parineq.etco}
Let $(A; G, H)$ be a parallelogram configuration over a function field $K$ of positive characteristic $p$.  Assume that $G$ or $H$ is \'etale of order a power of $p$. Then the parallelogram inequality \eqref{eqparallinequality} holds for $(A; G, H)$.
\end{lemm}
\begin{proof} By symmetry, we assume that $H$ is \'etale of order a power of $p$. The Cartier dual $H^\vee$ of $H$ is then a finite connected subgroup scheme of the dual abelian variety $A^\vee$.
By duality, the configuration $(A; G, H)$ gives rise to a commutative diagram 
\[ 
\begin{tikzcd}
& A^\vee  &    \\
(A/G)^\vee \arrow[ru, "\widehat{\varphi_G}"] &      & (A/H)^\vee \arrow[lu, "\widehat{\varphi_H}"'] \\
 & (A/(G+H))^\vee \arrow[lu, "\widehat{\psi_G}"] \arrow[ru, "\widehat{\psi_H}"'] &                                    
\end{tikzcd}.\]
Let $A':=(A/(G+H))^\vee$, $G':=\ker\widehat{\psi_G}$ and $H':=\ker\widehat{\psi_H}$.
One checks that $G'\cap H'$ is trivial, and that $A^\vee$ is isomorphic to $A'/(G'+H')$. In other words, the above diagram 
comes from the parallelogram configuration $(A'; G', H')$: 
\[
\begin{tikzcd}
                 & A'/(G'+H')                                        &                  \\
A'/G' \arrow[ru] &                                                   & A'/H' \arrow[lu] \\
                 & A' \arrow[lu, "\nu_{G'}"] \arrow[ru, "\nu_{H'}"'] &                 
\end{tikzcd}.\]
In this diagram, the isogeny $\nu_{G'} : A'\to A'/G'$ has kernel $G' = \ker\widehat{\psi_G}$, which is Cartier dual to $\ker\psi_G$. But, since $G$ and $H$ have trivial intersection, $\ker\psi_G = (G+H)/G$ is isomorphic to $H$. Hence $\nu_{G'}$ has connected kernel, and we may apply the previous Lemma~\ref{lemm:parineq.co} to the configuration $(A'; G', H')$. This yields
\[\hd(A') + \hd(A'/(G'+H')) \leq \hd(A'/G') + \hd(A'/H').\]
The definition of $A'$  
and the three isomorphisms $A'/G' \simeq (A/G)^\vee$, $ A'/H' \simeq (A/H)^\vee$, and $A'/(G'+H') \simeq A^\vee$
allow to rewrite the above inequality as
\[\hd(A^\vee) + \hd((A/(G+H))^\vee) \leq \hd((A/G)^\vee) + \hd((A/H)^\vee).\]
Remembering that a semi-stable abelian variety and its dual have the same height (Proposition~\ref{prop:hdiff.dual}) gives the parallelogram inequality for the configuration $(A; G, H)$.  
 \end{proof}

\subsection{Proof of Theorem~\ref{theo:parallelogram.ineq}}\label{sec:parineq.proof}
As noted above, it suffices to prove the parallelogram inequality \eqref{eqparallinequality} for any parallelogram configuration $(A;G,H)$ over $K$.
If $K$ has characteristic $0$, we have already done so in Lemma~\ref{lemm:parineq.etet}.
We may thus assume that $K$ has positive characteristic $p$.  

Let $(A;G,H)$ be a parallelogram configuration over $K$. 
The group scheme $H$ admits a filtration by subgroup schemes: 
\[ 0 \subset H^\circ \subset H_{p^\infty} \subset H,\] 
where $H^\circ$ is the connected component of identity in $H$, and $H_{p^\infty}$ is the $p$-primary part of $H$ (the largest subgroup scheme of $H$ whose order is a power of $p$). 
We consider the commutative diagram:
    \newcommand{\loz}[1]{\raisebox{-5pt}{\rotatebox[origin=c]{-90}{\scaleobj{3}{\lozenge}}}\hspace{-1.45em}#1\hspace{1em}}
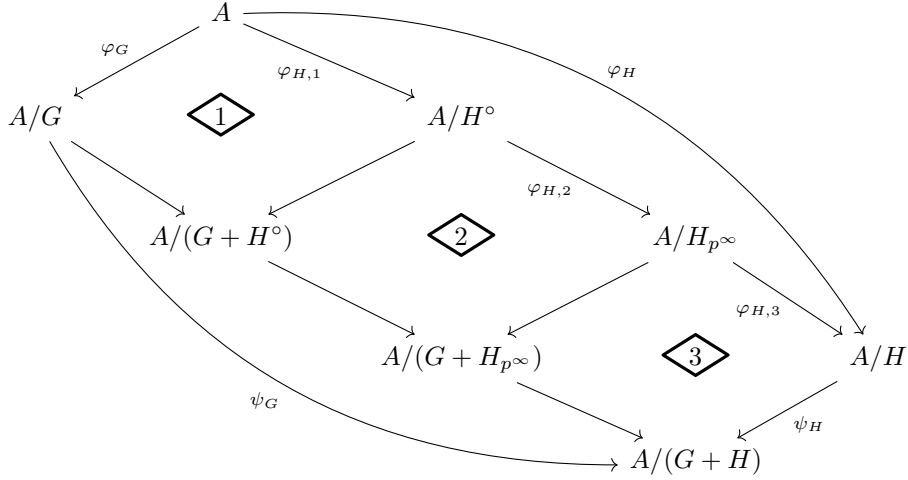
\[
\begin{tikzcd}
& A \arrow[rd, "\varphi_{H, 1}"'] \arrow[ld, "\varphi_G"'] \arrow[rrrddd, "\varphi_{H}", bend left] &  && \\
A/G \arrow[rd] \arrow[rrrddd, "\psi_G"', bend right] &  \loz{1}  & A/H^\circ \arrow[rd, "\varphi_{H, 2}"'] \arrow[ld] &  & \\
 & A/(G+H^\circ) \arrow[rd] &   \loz{2} & A/H_{p^\infty}\arrow[rd, "\varphi_{H,3}"'] \arrow[ld] &  \\
 & & A/(G+H_{p^\infty}) \arrow[rd] & \loz{3} & A/H \arrow[ld, "\psi_H"]   \\
 &  &  & A/(G+H)  &                                                     
\end{tikzcd}\]
In the parallelogram labeled $\loz{1}$, the isogeny $\varphi_{H, 1}$ has connected kernel $H^\circ$. 
Lemma~\ref{lemm:parineq.co} yields the parallelogram inequality for $(A;G,H^\circ)$.
In the parallelogram labeled $\loz{2}$, the isogeny $\varphi'_{H, 2}$ has kernel isomorphic to $H_{p^\infty}/H^\circ$, which is \'etale of order a power of $p$.  
The parallelogram inequality holds for $(A/H^\circ;(G+H^\circ)/H^\circ,H_{p^\infty}/H^\circ)$ by Lemma~\ref{lemm:parineq.etco}.
In the parallelogram labeled $\loz{3}$, the isogeny $\varphi'_{H, 3}$ has  kernel isomorphic to $H/H_{p^\infty}$, which has order coprime to $p$ so it is \'etale with \'etale dual.  
The parallelogram inequality for $(A/H_{p^\infty};(G+H_{p^\infty})/H_{p^\infty},H/H_{p^\infty})$ then follows from Lemma~\ref{lemm:parineq.etet}.

Finally, three successive applications of Lemma~\ref{lemm:parineq.devissage} imply that the parallelogram inequality \eqref{eqparallinequality} holds for the  configuration $(A;G, H)$. 
\hfill$\Box$

\subsection{Consequence: height of abelian subvarieties}

\begin{theo}\label{theo:height.subvar}
Let $K=k(C)$ be a function field as above. Let $A$ be an abelian variety over $K$.
If $K$ has positive characteristic, assume that $A$ is semi-stable. For any abelian subvariety $B\subset A$ of $A$, we have
\begin{equation}\label{eq:height.subvar}
\hd(B) + \hd(A/B) \leq\hd(A).    
\end{equation}
Moreover, in case $K$ has characteristic $0$, this inequality is an equality.
\end{theo}

\begin{proof}
We essentially rewrite R\'emond's proof of th\'eor\`eme 1.2 in \cite{R22}  in our context.
Choose a quasi-complement $B'$ of $B$ in $A$ \ie{}, an abelian subvariety $B'\subset A$ such that $B+B'=A$, and $B\cap B'$ finite.
 Let $G$ denote the kernel of the summation isogeny $B\times B' \to A$, and let $H := (B\cap B')\times \{0_{B'}\}$, where $0_{B'}$ denotes the neutral element in $B'$.
Both $G$ and $H$ are finite subgroup schemes of $B\times B'$, and one easily shows that 
\[G\cap H \text{ is trivial}, \quad \text{ and }\quad G+H=(B\cap B')\times(B\cap B').\]
It is clear that $(B\times B')/G \simeq A$, and that $(B\times B')/H \simeq (B/B\cap B')\times B'$.
The isogeny $B' \to A \to A/B$ has kernel $B\cap B'$, so that $B'/(B\cap B')\simeq A/B$ and therefore $(B\times B')/(G+ H) \simeq (B/B\cap B')\times A/B$. 

 Applying Theorem~\ref{theo:parallelogram.ineq} to the parallelogram configuration $(B\times B'; G, H)$ yields
 \[\hd(B\times B') +  \hd((B\times B')/(G+ H)) \leq \hd((B\times B')/G) + \hd((B\times B')/H).\]
The differential height of a product is the sum of the differential heights (see Proposition~\ref{prop:hdiff.dual}) so, with the above isomorphisms the previous inequality (which is an equality if $K$ has characteristic $0$), becomes
\[\hd(B)+\hd(B') + \hd(B/(B\cap B')) + \hd(A/B) \leq \hd(A) +\hd(B/(B\cap B')) + \hd(B').\]
Canceling terms yields the desired result.  
\end{proof}

\begin{rema} In positive characteristic and in the ordinary case, the less precise inequality $\hd(B)\leq\hd(A)$ is easier to prove, and does not require the use of Theorem~\ref{theo:parallelogram.ineq}.
Indeed, write $\Bcal$ and $\Acal$ for the respective N\'eron models of $B$ and $A$. 
Then ${\Gcal := \Bcal[\Frob_{\Bcal/C}]}$ is a finite flat subgroup scheme of $\Acal[\Frob_{\Acal/C}]$ with Frobenius height $1$.
Recall that $\hd(B) = -\deg\Lie(\Bcal/C) = -\deg\Lie(\Gcal/C)$. 
Inequality \eqref{eq:keyineq} in the proof of Proposition~\ref{prop:hdiff.finite.depth} yields
$-\deg\Lie(\Gcal/C) \leq \hd(A)$, which is the claimed inequality.

Notice also that, if $A$ a constant abelian variety over $K$, then any abelian subvariety $B$ of $A$ is also constant over~$K$ 
(by \cite[Corollary 3.21]{Conrad}). This is consistent with the inequality $\hd(B) \leq \hd(A)= 0$.
\end{rema}

\begin{exple}\label{exple:strict.ineq.subvar}
If $K$ has positive characteristic $p$, the inequality in Theorem~\ref{theo:height.subvar} cannot always be an equality (\ie{}, it must sometimes be strict), as the following example shows.
Let $E_1$ and $E_2$ be two non-isotrivial ordinary elliptic curves over $K$.
Assume that the differential heights of $E_1$ and $E_2$ are different.  
Up to replacing $K$ by a finite extension, we may assume that both $E_1$ and $E_2$ have a $K$-rational point of exact order $p$. 
For $i=1, 2$, pick such a point $P_i\in E_i(K)$ and write $\langle P_i\rangle$ for the subgroup of $E_i$ it generates. 
The quotient isogeny $E_i\to E_i/\langle P_i\rangle$ is \'etale of degree $p$, so that 
$\hd(E_i/\langle P_i\rangle) = p\, \hd(E_i)$. 

Now, let $P=(P_1, P_2)$ be the induced $p$-torsion point on the product $E_1\times E_2$, and 
consider the quotient abelian surface $A := (E_1\times E_2)/\langle P\rangle$. 
For $i=1, 2$, the natural  morphism $f_i:E_i \hookrightarrow E_1\times E_2 \to A$ is injective. Let $B_i := f_i(E_i) \subset A$ denote its image.
By construction, $B_i$ is isomorphic to $E_i$. Moreover, a straightforward computation shows that the quotient $A/B_1$ is isomorphic to $E_2/\langle P_2\rangle$ and, similarly, that
$A/B_2 \simeq E_1/\langle P_1\rangle$.

If inequality \eqref{eq:height.subvar} were an equality for both subvarieties  $B_1$ and $B_2$ of $A$, then we would have 
\begin{align*}
    \hd(A) &= \hd(B_1) + \hd(A/B_1) = \hd(E_1) + \hd(E_2/\langle P_2\rangle)  = \hd(E_1) + p\, \hd(E_2), \\
\text{ and }\  \hd(A) &= \hd(B_2) + \hd(A/B_2) = \hd(E_2) + \hd(E_1/\langle P_1\rangle)  = p\,\hd(E_1) +  \hd(E_2).
\end{align*}
This would lead to $\hd(E_1)=\hd(E_2)$, contradicting our choice of $E_1$ and $E_2$. 
Therefore, one of the inequalities $\hd(B_1) + \hd(A/B_1)\leq\hd(A)$ and $\hd(B_2) + \hd(A/B_2)\leq\hd(A)$ must be strict.
\end{exple}

\begin{rema}\label{rema:equiv.parineq.sub}
R\'emond remarked in private communication that Theorem  \ref{theo:height.subvar} and Theorem \ref{theo:parallelogram.ineq} are actually equivalent statements. 
To prove that  Theorem \ref{theo:height.subvar} implies Theorem \ref{theo:parallelogram.ineq}, notice that any parallelogram triple $(A;G,H)$ gives rise to an exact sequence of abelian varieties
    \[0\longrightarrow A/(G\cap H) \xrightarrow[]{\ \iota\ } A/G \times A/H \longrightarrow A/(G+H)\longrightarrow 0,\]
     where $\iota$ is induced by the diagonal embedding $G\cap H\into G\times H$.
    Applying Theorem \ref{theo:height.subvar} to the abelian subvariety $\iota(A/(G\cap H))$ of $A/G\times A/H$, 
    and using a classical isomorphism theorem and the compatibility of the height with products yield the parallelogram inequality \eqref{eqparallinequality} for the triple $(A; G, H)$. 
\end{rema}


\begin{center}
    \rule{10cm}{1pt}
\end{center}


\paragraph{Acknowledgments --} 
The authors wish to thank Michel Brion, Francesco Campagna, and Marc Hindry for interesting conversations and useful feedback. 
The authors give their warmest thanks to Ga\"el R\'emond for his careful reading of a draft of this paper and his fruitful suggestions.   
At various stages of the project, the authors have received financial support from 
CNRS in the form of PEPS JCJC grants,  IRN MADEF,  and  IRN GandA, 
from ANR projects {\it Jinvariant} (ANR-20-CE40-0003) and
{\it GAEC} (ANR-23-CE40-0006-01), 
 and from IRGA PointRatMod. 
 They express their gratitude to these institutions for their continued support.
FP thanks the Universit\'e de Bordeaux for its hospitality.

\begin{center}
    \rule{10cm}{1pt}
\end{center}


	\normalsize\vfill
	\noindent\rule{7cm}{0.5pt}
	
	\smallskip
	\noindent
	{Richard {\sc Griffon}} {(\it \href{richard.griffon@uca.fr}{richard.griffon@uca.fr})}  --
	{\sc Laboratoire de Math\'ematiques Blaise Pascal, Universit\'e Clermont Auvergne}. 
	3 Place Vasarely, 63710 Aubi\`ere (France).
	
	\medskip
	
	\noindent
	{Samuel {\sc Le Fourn}} {(\it \href{samuel.le-fourn@univ-grenoble-alpes.fr}{samuel.le-fourn@univ-grenoble-alpes.fr})}  --
	{\sc Institut Fourier, Universit\'e Grenoble Alpes}. 
	100 rue des maths, 38610 Gières (France).
	
	\medskip
	
	\noindent
	{Fabien {\sc Pazuki}} {(\it \href{fpazuki@math.ku.dk}{fpazuki@math.ku.dk})} --
	{\sc Department of Mathematical Sciences, University of Copenhagen}.
	Universitetsparken 5, 2100 Copenhagen \o{} (Denmark).
				
\end{document}